\documentclass[12pt]{amsart}

\usepackage[english]{babel}
\usepackage[utf8x]{inputenc}
\usepackage[T1]{fontenc}

\usepackage[margin=1.0in]{geometry}

\usepackage{amsmath}
\usepackage{amstext}
\usepackage{amsfonts}
\usepackage{amssymb, dsfont}
\usepackage{amsthm}
\usepackage{amsrefs}
\usepackage{mathtools}
\usepackage{color}
\usepackage{graphicx}
\usepackage[colorlinks=true, allcolors=blue]{hyperref}
\usepackage{float}
\allowdisplaybreaks
\usepackage{tikz-cd}
\usepackage{amsbsy}
\usepackage{amscd}
\usetikzlibrary{matrix,arrows,decorations.pathmorphing}
\usepackage{enumitem}
\usepackage{setspace}

\usepackage{amsrefs}

\newcommand\numberthis{\addtocounter{equation}{1}\tag{\theequation}}

\newtheorem{thm}{Theorem}[section]
\newtheorem{lem}[thm]{Lemma}
\newtheorem{cor}[thm]{Corollary}
\newtheorem{prop}[thm]{Proposition}

\theoremstyle{definition}

\theoremstyle{definition}

\theoremstyle{definition}
\newtheorem{defn}[thm]{Definition}
\theoremstyle{definition}
\newtheorem{remark}[thm]{Remark}

\newcommand{\mc}[1]{\mathcal{#1}}
\newcommand{\e}[1]{\emph{#1}}

\newcommand{\la}{\langle}
\newcommand{\ra}{\rangle}

\newcommand{\rmv}[1]{}

\newcommand{\LO}{L^1(G)}

\newcommand{\LT}{L^2(G)}

\newcommand{\LI}{L^{\infty}(G)}

\newcommand{\BH}{\mc{B}(H)}

\newcommand{\vphi}{\varphi}

\newcommand{\al}{\alpha}
\newcommand{\wtal}{\widetilde{\alpha}}

\newcommand{\lm}{\lambda}

\newcommand{\om}{\omega}

\newcommand{\ten}{\otimes}
\newcommand{\oten}{\overline{\otimes}}
\newcommand{\pten}{\ten^{\pi}}
\newcommand{\opten}{\widehat{\ten}}
\newcommand{\hten}{\otimes^h}

\newcommand{\whten}{\otimes^{w^*h}}

\newcommand{\id}{\textnormal{id}}

\newcommand{\Ad}{\mathrm{Ad}}

\newcommand{\ep}{\varepsilon}

\providecommand{\norm}[1]{\lVert#1\rVert}

\newcommand{\supp}{\operatorname{supp}}

\newcommand{\bC}{{\mathbb{C}}}

\renewcommand{\P}{{\mathcal{P}}}

\newcommand{\acts}{\curvearrowright}

\begin{document}

\title[]{Amenable dynamical systems over locally compact groups}
\author{Alex Bearden}
\email{cbearden@uttyler.edu}
\address{Department of Mathematics, University of Texas at Tyler, Tyler, TX 75799}
\author{Jason Crann}
\email{jasoncrann@cunet.carleton.ca}
\address{School of Mathematics and Statistics, Carleton University, Ottawa, ON K1S 5B6}



\begin{spacing}{1.0}

\maketitle
\begin{abstract} We establish several new characterizations of amenable $W^*$- and $C^*$-dynamical systems over arbitrary locally compact groups. In the $W^*$-setting we show that amenability is equivalent to (1) a Reiter property and (2) the existence of a certain net of completely positive Herz-Schur multipliers of $(M,G,\alpha)$ converging point weak* to the identity of $G\bar{\ltimes}M$. In the $C^*$-setting, we prove that amenability of $(A,G,\alpha)$ is equivalent to an analogous Herz-Schur multiplier approximation of the identity of the reduced crossed product $G\ltimes A$, as well as a particular case of the positive weak approximation property of B\'{e}dos and Conti \cite{BC3} (generalized the locally compact setting). When $Z(A^{**})=Z(A)^{**}$, it follows that amenability is equivalent to the 1-positive approximation property of Exel and Ng \cite{EN}. In particular, when $A=C_0(X)$ is commutative, amenability of $(C_0(X),G,\alpha)$ coincides with topological amenability the $G$-space $(G,X)$. Our results answer 2 open questions from the literature; one of Anantharaman--Delaroche from \cite{AD87}, and one from recent work of Buss--Echterhoff--Willett \cite{BEW}.
\end{abstract}

\section{Introduction} 

Amenability and its various manifestations have played an important role in the study of dynamical systems and their associated operator algebras. Zimmer introduced a dynamical version of amenability \cite{Zim} of an action of a locally compact group on a standard measure space through a generalization of Day's fixed point criterion, which has proven very useful in ergodic theory and von Neumann algebras. 

Motivated by the structure of crossed products, Anantharaman-Delaroche generalized Zimmer's notion of amenability to the level of $W^*$-dynamical systems $(M,G,\alpha)$ \cite{ADI}. In \cite[Th\'{e}or\`{e}me 3.3]{AD87} she characterized amenability of $(M,G,\alpha)$ with $G$ discrete through a Reiter type property involving asymptotically $G$-invariant functions in $C_c(G,M)$, generalizing Reiter's condition for amenable groups. She also introduced a notion of amenability for discrete $C^*$-dynamical systems $(A,G,\alpha)$, and showed, among other things, that a commutative discrete $C^*$-dynamical system $(C_0(X),G,\alpha)$ is amenable precisely when the transformation groupoid $G\ltimes X$ is topologically amenable in the sense of Renault \cite{Renault}.

Various approaches to amenability for non-discrete $C^*$-dynamical systems have been studied, including amenable transformation groups (e.g., \cite{AD}) and the approximation property of Exel and Ng \cite{EN}. Recently, a notion of amenability for arbitrary $C^*$-dynamical systems was introduced by Buss, Echterhoff and Willett \cite{BEW}, who performed an in-depth study of this notion in relation to amenability of the universal $W^*$-dynamical system \cite{I}, measurewise amenability, and the weak containment problem (among other things).

In this work we establish several new characterizations of amenable $W^*$- and $C^*$-dynamical systems over arbitrary locally compact groups. For $W^*$-systems we generalize \cite[Th\'{e}or\`{e}me 3.3]{AD87} to the locally compact setting, giving a Reiter property for arbitrary amenable $(M,G,\alpha)$ (see Theorem \ref{t:Reiter}). Our approach relies on a continuous version of \cite[Lemme 3.1]{AD87}, whose validity was asked by Anantharaman--Delaroche in that paper. We therefore answer this question in the affirmative. We also characterize amenability of arbitrary $(M,G,\alpha)$ through a ``fundamental unitary'' $W_\alpha$ associated to the action, and through Herz-Schur multipliers on the crossed product $G\bar{\ltimes}M$ \cite{BC,MTT,MSTT}. Our results in this context can be summarized by:

\begin{thm}\label{t:introW*} Let $(M,G,\alpha)$ be a $W^*$-dynamical system with $M\subseteq\BH$. The following conditions are equivalent:
\begin{enumerate}
\item $(M,G,\alpha)$ is amenable; 
\item there exists a net $(\xi_i)$ in $C_c(G,Z(M)_c)$ such that
\begin{enumerate}
\item $\la\xi_i,\xi_i\ra=1$ for all $i$;
\item $\la\xi_i, (\lm_s\ten\alpha_s)\xi_i\ra\rightarrow 1$ weak*, uniformly on compact subsets of $G$.
\end{enumerate}
\item there exists a net $(\xi_i)$ in $C_c(G,Z(M)_c)$ such that
\begin{enumerate}
\item $\la\xi_i,\xi_i\ra=1$ for all $i$;
\item $\norm{W_\alpha(\xi_i\ten_\alpha\eta)-\xi_i\cdot\eta}_{L^2(G\times G,H)}\rightarrow0$, $\eta\in L^2(G,H)$;
\end{enumerate}
\item there exists a net $(\xi_i)$ in $C_c(G,M_c)$ such that
\begin{enumerate}
\item $\la\xi_i,\xi_i\ra=1$ for all $i$;
\item $\Theta(h_{\xi_i})\rightarrow\id_{G\bar{\ltimes} M}$ point weak*,
\end{enumerate}
where $h_{\xi_i}(s)(a)=\la\xi_i,(1\ten a)(\lm_s\ten\alpha_s)\xi_i\ra$, are the associated completely positive Herz-Schur multipliers in the sense of \cite{MTT}, and $\Theta(h_{\xi_i})$ are the induced mappings on $G\bar{\ltimes}M$. 
\end{enumerate}
\end{thm}

The equivalence between (1) and (4) in Theorem \ref{t:W*amen} may be viewed as a dynamical systems analogue of \cite[Theorem 1.13]{HK}, which characterizes amenability of a locally compact group $G$ through a net $(u_i)$ of normalized positive definite functions on $G$ whose multipliers converge to the identity of $VN(G)$ in the point weak* topology.

For $C^*$-dynamical systems, we complement the recent work of Buss, Echterhoff and Willett \cite{BEW} by showing the equivalence between their notion of amenability, amenability of the universal enveloping $W^*$-system, and a particular case of the 1-positive weak approximation property of B\'{e}dos and Conti \cite{BC3} (suitably generalized to the locally compact setting). We also obtain an analogous Herz-Schur multiplier characterization at the level of the reduced crossed product. Our results in this context are summarized by:

\begin{thm}\label{t:intro} Let $(A,G,\alpha)$ be a $C^*$-dynamical system. The following conditions are equivalent:
\begin{enumerate}
\item $(A,G,\alpha)$ is amenable in the sense of \cite{BEW};
\item there exists a net $(\xi_i)$ in $C_c(G,\ell^2(A))$ such that
\begin{enumerate}
\item $\la\xi_i,\xi_i\ra\leq $ for all $i$;
\item $h_{\xi_i}(e)\rightarrow\id_{A}$ in the point norm topology, and
\item $\Theta(h_{\xi_i})\rightarrow\id_{G\ltimes A}$ in the point norm topology,
\end{enumerate}
where $h_{\xi_i}(s)(a)=\la\xi_i,(1\ten 1\ten a)(\lm_s\ten1\ten \alpha_s)\xi_i\ra$ are the associated completely positive Herz-Schur multipliers in the sense of \cite{MTT}, and $\Theta(h_{\xi_i})$ are the induced mappings on $G\ltimes A$;
\item there exists a net $(\xi_i)$ in $C_c(G,\ell^2(A))$ such that $\la\xi_i,\xi_i\ra\leq1 $ for all $i$, and 
$$\norm{h_{\xi_i}(f(s))-f(s)}\rightarrow0, \ \ \ f\in C_c(G,A),$$
uniformly for $s$ in compact subsets of $G$;
\item the universal $W^*$-dynamical system $(A_\alpha'',G,\overline{\alpha})$ (from \cite{I}) is amenable.
\end{enumerate}
Moreover, when $Z(A^{**})=Z(A)^{**}$, the net $(\xi_i)$ can be chosen in $C_c(G,Z(A))$, in which case $h_{\xi_i}(s)(a)=a\la\xi_i,(\lm_s\ten\alpha_s)\xi_i\ra$, $s\in G$, $a\in A$.
\end{thm}

The equivalence $(1)\Leftrightarrow(4)$ generalizes the corresponding result for exact locally compact groups \cite[Proposition 3.12]{BEW}. 

As a corollary to Theorem \ref{t:intro}, when $Z(A^{**})=Z(A)^{**}$, amenability of $(A,G,\alpha)$ is equivalent to the 1-positive approximation property of Exel and Ng \cite{EN}. This partially answers the recently posed \cite[Question 8.2]{BEW}. It follows that a commutative $C^*$-dynamical system $(C_0(X),G,\alpha)$ is amenable in the sense of \cite{BEW} if and only if the transformation group $(G,X)$ is topologically amenable (see Corollary \ref{c:SA}). This generalizes \cite[Th\'{e}or\`{e}me 4.9]{AD87} from discrete groups to arbitrary locally compact groups, and answers \cite[Question 8.1]{BEW} in the affirmative. Combining Corollary \ref{c:SA} with the recent result \cite[Theorem 5.16]{BEW} of Buss, Echterhoff and Willett, we obtain a positive answer to the long standing open question whether topological amenability and measurewise amenability coincide for actions $G\acts X$ when $G$ and $X$ are second countable. 


The paper is outlined as follows. We begin in section 2 with preliminaries on dynamical systems and vector-valued integration. Section 3 contains our results on amenable $W^*$-dynamical systems as well as results of independent interest which build on the recent theory of Herz-Schur multipliers for crossed products \cite{BC,BC2,MTT,MSTT}. Section 4 contains our results on amenable $C^*$-dynamical systems. 

\section{Preliminaries}

\subsection{Vector-Valued Integration}
Throughout this subsection $S$ will be a locally compact Hausdorff space with positive Radon measure $\mu$. 

For a Banach space $B$, we let $L^1(S,B)$ denote the space of (locally a.e.\ equivalence classes of) Bochner integrable functions $f: S \to B$ with the norm $\|f\| = \int_S \|f\| \, d\mu(s)$. By the Pettis Measurability Theorem and Bochner's Theorem (see \cite[Section 2.3]{Ryan}), for $f: S \to B$ supported on a $\sigma$-finite set, $f \in L^1(S,B)$ if and only if $f$ is weakly measurable, essentially separably valued, and satisfies $\int_S \|f(s)\| \, d\mu(s) < \infty$. In particular, there is a canonical map $C_c(S,B)\to L^1(S,B)$, where $C_c(S,B)$ denotes the continuous $B$-valued functions of compact support. It is well-known that $L^1(S,B)\cong L^1(S,\mu)\pten B$ isometrically, where $\pten$ is the Banach space projective tensor product (see, e.g., \cite[Proposition IV.7.14]{Takesaki}).

If $M$ is a von Neumann algebra we have the following canonical identifications: \[(L^\infty(S,\mu) \overline{\otimes} M)_* \cong L^1(S,M_*) \cong L^1(S,\mu) \otimes^\pi M_*.\] (See \cite[Proposition IV.7.14 and Theorem IV.7.17]{Takesaki}.) We remark that $L^\infty(S,\mu) \overline{\otimes} M$ does not necessarily coincide with the space $L^\infty(S,M)$ of essentially bounded $w^*$- locally measurable functions from $S$ to $M$ since we do not assume that $M_*$ is separable (see \cite{S,Takemoto}). However, by \cite[Theorem IV.7.17]{Takesaki}, for each $F \in L^\infty(S) \overline{\otimes} M$, there exists a weak*-measurable function $\tilde{F}:S \to M$ such that for every $g \in L^1(S,M_*)$, the function $s \mapsto \langle \tilde{F}(s),g(s) \rangle$ is a measurable function on $S$, and \[ \langle F,g \rangle = \int_S \langle \tilde F(s),g(s) \rangle \, d\mu(s), \qquad  g \in L^1(S,M_*).\] In this case, we will say that $\tilde{F}$ \emph{represents $F$}, and usually abuse notation by omitting the tilde in the latter centered equation. There are some pitfalls that one must take care to avoid though---for example, if $S=[0,1]$ with Lebesgue measure, and $M = \ell^\infty[0,1]$ is the space of all bounded functions on $[0,1]$, then the function $f: S \to M$, $f(t) = \chi_{\{t\}}$, is nonzero everywhere, but $f$ represents $0 \in L^\infty(S) \overline{\otimes} M$.

\begin{lem}\label{l:1}
If $M$ is a von Neumann algebra and $\omega \in M_*$, there is a map $\tilde \omega: L^1(S,M) \to L^1(S,M_*)$ determined by the formula \[ \langle \tilde \omega(g)(s),x \rangle = \langle \omega, g(s)x  \rangle\] for $g \in L^1(S,M)$, $s \in S$, and $x \in M$. Moreover, $\norm{\tilde{\om}}\leq\norm{\om}$.
\end{lem}

\proof
Using the canonical identifications, the map $\tilde \omega$ is just $\mathrm{id} \otimes \omega_0: L^1(S) \otimes^{\pi} M \to L^1(S) \otimes^{\pi} M_*$, where $\omega_0: M \to M_*$ is the operator satisfying $\langle \omega_0(y), x \rangle = \langle \omega, yx \rangle$ for $x,y \in M$. The norm inequality is obvious.
\endproof

If $A$ is a $C^*$-algebra, we let $L^2(S,A)$ denote the Hilbert module completion of $C_c(S,A)$ under the $A$-valued inner product
$$\la\xi,\eta\ra=\int_S\xi(s)^*\eta(s) \ d\mu(s), \ \ \ \xi,\eta\in C_c(S,A).$$

\subsection{Dynamical Systems} \label{subsection: ds}
A $W^*$-dynamical system $(M,G,\alpha)$ consists of a von Neumann algebra $M$ endowed with a homomorphism $\alpha:G\rightarrow\mathrm{Aut}(M)$ of a locally compact group $G$ such that for each $x\in M$, the map $G\ni s\rightarrow \alpha_s(x)\in M$ is weak* continuous. In this case, the canonical action $G \acts M_*$ is norm-continuous (see \cite[Proposition 1.2']{Takesaki2}). We let $M_c$ denote the unital $C^*$-subalgebra consisting of those $x\in M$ for which $s\mapsto \alpha_s(x)$ is norm continuous. By \cite[Lemma 7.5.1]{Ped}, $M_c$ is weak* dense in $M$.

The action $\alpha$ induces a normal injective unital $*$-homomorphism 
$$\alpha:M\ni x\rightarrow(s\mapsto \alpha_{s^{-1}}(x))\in\LI\oten M$$
defined by \[\la \alpha(x),F \ra = \int_G \la \alpha_{s^{-1}}(x),F(s) \ra \, ds, \quad \text{for } F \in L^1(G,M_*).\]
A normal covariant representation $(\pi, u)$ of $( M,G,\alpha)$ consists of a normal representation $\pi: M\rightarrow\BH$ and a unitary representation $u:G\rightarrow\BH$ such that $\pi(\alpha_s(x))=u_s\pi(x)u_{s^{-1}}$ for all $x \in M$, $s\in G$.
When $(\pi,u)$ is a normal covariant representation of $(M,G,\alpha)$ (this includes the case when $M\subseteq\BH$ is standardly represented, since in this case there exists a unique strongly continuous unitary representation $u:G\rightarrow \BH$ such that $\alpha_s(x) = u_s x u_{s^{-1}}$ by \cite[Corollary 3.6]{H}), there is corresponding generator $U\in\LI\oten\BH$, defined
\[ \la U,F \ra = \int_G \la u_s,F(s) \ra, \quad \text{for } F \in L^1(G,M_*),\] and we have $\alpha(x)=U^*(1\ten x)U$, $x\in M$. Moreover, for any $\xi\in L^2(G,H)$,
\begin{align*}U(\lm_s\ten 1)\xi(t)&=u_t((\lm_s\ten 1)\xi(t))=u_t(\xi(s^{-1}t))\\
&=u_su_{s^{-1}t}(\xi(s^{-1}t))\\
&=u_s(U\xi(s^{-1}t))\\
&=(\lm_s\ten u_s)U\xi(t).
\end{align*}
Hence, $U(\lm_s\ten 1)=(\lm_s\ten u_s)U$ for any $s\in G$.

A $C^*$-dynamical system $(A,G,\alpha)$ consists of a $C^*$-algebra endowed with a homomorphism $\alpha:G\rightarrow\mathrm{Aut}( A)$ of a locally compact group $G$ such that for each $a\in A$, the map $G\ni s\mapsto\alpha_s(a)\in A$ is norm continuous. 

A covariant representation $(\pi, \sigma)$ of $( A,G,\alpha)$ consists of a representation $\pi: A\rightarrow\BH$ and a unitary representation $\sigma:G\rightarrow\BH$ such that $\pi(\alpha_s(a))=\sigma_s\pi(a)\sigma_{s^{-1}}$ for all $a \in A$, $s\in G$. Given a covariant representation $(\pi,\sigma)$, we let
$$(\pi \times \sigma)(f) = \int_G \pi(f(t)) \sigma_t \, dt, \ \ \ f\in C_c(G, A).$$
The full crossed product $G\ltimes_f A$ is the completion of $C_c(G, A)$ in the norm
$$\|f \| = \sup_{(\pi, \sigma)} \| (\pi \times \sigma)(f)\|$$
where the $\sup$ is taken over  all covariant representations $(\pi, \sigma)$ of $(A,G,\alpha)$.

Let $A\subseteq\mc{B}(H)$ be a faithful non-degenerate representation of $ A$. Then $(\alpha,\lm\ten 1)$ is a covariant representation on $L^2(G,H)$, where 
$$\alpha(a)\xi(t)=\alpha_{t^{-1}}(a)\xi(t), \ \ \ (\lm\ten 1)(s)\xi(t)=\xi(s^{-1}t), \ \ \ \xi\in L^2(G,H).$$
The reduced crossed product $G\ltimes A$ is defined to be the norm closure of $(\alpha\times(\lm\ten 1))(C_c(G,A))$. This definition is independent of the faithful non-degenerate representation $A\subseteq\mc{B}(H)$. We often abbreviate $\alpha\times(\lm\ten 1)$ as $\alpha\times\lm$. Recall that $C_c(G,A)$ is a $*$-algebra under the operations
$$f\star g(s)=\int_G f(t)\alpha_t(g(t^{-1}s) \ dt, \ \ \ f^*(s)=\Delta(s^{-1})\alpha_s(f(s^{-1})^*), \ \ \ f,g\in C_c(G,A),$$
and that $\alpha\times\lm$ is a $*$-homomorphism. 

Analogous to the group setting, dual spaces of crossed products can be identified with certain $ A^*$-valued functions on $G$. We review aspects of this theory below and refer the reader to \cite[Chapters 7.6, 7.7]{Ped} for details. 

For each $C^*$-dynamical system $( A,G,\alpha)$ there is a universal covariant representation $(\pi,\sigma)$ such that 
$$G\ltimes_f A\subseteq C^*(\pi( A)\cup \sigma(G))\subseteq M(G\ltimes_f A).$$
Each functional $\vphi\in (G\ltimes_f A)^*$ then defines a function $\Phi:G\rightarrow A^*$ by
\begin{equation} \label{eqn: FS relation} \la\Phi(s),a\ra=\vphi(\pi(a)\sigma_s), \ \ \ a\in A, \ s\in G.\end{equation}
Let $B(G\ltimes_f A)$ denote the resulting space of $ A^*$-valued functions on $G$. An element $\Phi\in B(G\ltimes_f A)$ is \textit{positive definite} if it arises from a positive linear functional $\vphi$ as above. We let $A(G\ltimes_f A)$ denote the subspace of $B(G\ltimes_f A)$ whose associated functionals $\vphi$ are of the form
$$\vphi(x)=\sum_{n=1}^\infty\la\xi_n, \alpha\times\lm(x)\eta_n\ra, \ \ \ x\in G\ltimes_f  A,$$
for sequences $(\xi_n)$ and $(\eta_n)$ in $L^2(G,H)$ with $\sum_{n=1}^\infty \norm{\xi_n}^2<\infty$ and $\sum_{n=1}^\infty\norm{\eta_n}^2<\infty$. Then $A(G\ltimes_f A)$ is a norm closed subspace of $(G\ltimes_f A)^*$ which can be identified with $((G\ltimes A)'')_*$.

A function $h:G\rightarrow A$ is of positive type (with respect to $\alpha$) if for every $n\in\mathbb{N}$, and $s_1,...,s_n\in G$, we have
$$[\alpha_{s_i}(h(s_{i}^{-1}s_j)]\in M_n(A)^+.$$
We let $P_1(A,G,\alpha)$ denote the convex set of positive type functions with $\norm{h_i(e)}\leq 1$.


Every $C^*$-dynamical system $(A,G,\alpha)$ admits a unique universal $W^*$-dynamical system $(A_\alpha'',G,\overline{\alpha})$ \cite{I}. We review this construction taking an $\LO$-module perspective. In \cite{BEW}, they study $(A_\alpha'',G,\overline{\alpha})$ from a different, equivalent perspective.

First, $A$ becomes a right operator $\LO$-module in the canonical fashion by slicing the corresponding non-degenerate representation
$$\alpha:A\ni a \mapsto (s\mapsto\alpha_{s^{-1}}(a))\in C_b(G,A)\subseteq\LI\oten A^{**}.$$
Explicitly, this action is given by
\[ a \ast f = \int_G  f(s)\alpha_{s^{-1}}(a) \, ds\]
for $a \in A, \ f \in L^1(G)$.
By duality we obtain a left operator $\LO$-module structure on $A^*$ via
$$\alpha^*|_{\LO\opten A^*}:\LO\opten A^*\rightarrow A^*.$$
Then $G$ acts in a norm-continuous fashion on the essential submodule 
$$A^*_c:=\la\LO\ast A^*\ra,$$
where $\la\cdot\ra$ denotes closed linear span. The same argument in \cite[Lemma 7.5.1]{Ped} shows that $A^*_c$ coincides with the norm-continuous part of $A^*$ (i.e., the set of $\varphi \in A^*$ such that each map $G \to A^*$, $s \mapsto \varphi \circ \alpha_s$ is norm-continuous), hence the notation. This fact was also noted by Hamana in \cite[Proposition 3.4(i)]{Ham11}. We therefore obtain a point-weak* continuous action of $G$ on the dual space $(A^*_{c})^*$ by surjective complete isometries. Clearly 
\begin{equation}\label{e:iden}(A^*_{c})^*\cong A^{**}/(A^*_{c})^{\perp}\end{equation}
completely isometrically and weak*-weak* homeomorphically as right $\LO$-modules, where the canonical $\LO$-module structure on $A^{**}$ is obtained by slicing the normal cover of $\alpha$, which is the normal $*$-homomorphism
$$\widetilde{\alpha}=(\alpha^*|_{\LO\opten A^*})^*:A^{**}\rightarrow\LI\oten A^{**}.$$
Note that $\widetilde{\alpha}|_{M(A))}$ is the unique strict extension of $\alpha$, and is therefore injective \cite[Proposition 2.1]{Lance}. However, on $A^{**}$, $\widetilde{\alpha}$ can have a large kernel. On the one hand, its kernel is of the form $(1-z)A^{**}$ for some projection $z\in Z(A^{**})$. On the other hand, by definition of the $\LO$-action on $A^{**}$, $\mathrm{Ker}(\widetilde{\alpha})=(A^*_{c})^{\perp}$. It follows that $(A^*_{c})^*$ is completely isometrically weak*-weak* order isomorphic to $zA^{**}$, where we equip $(A^*_{c})^*$ with the quotient operator system structure from $A^{**}$. We can therefore transport the point-weak* continuous $G$-action on $(A^*_{c})^*$ to $A_\alpha'':=zA^{**}$, yielding a $W^*$-dynamical system $(A_\alpha'',G,\overline{\alpha})$, where $\overline{\alpha}:G\rightarrow\mathrm{Aut}(A_\alpha'')$ is given by
$$\overline{\alpha}_t(zx)=z((\alpha_t)^{**}(x)), \ \ \ x\in A^{**}, \  t\in G.$$
The associated normal injective $*$-homomorphism 
$$\overline{\alpha}:A_\alpha''\rightarrow\LI\oten A_\alpha''$$
is $(\id\ten\mathrm{Ad}(z))\circ\widetilde{\alpha}|_{A_\alpha''}$. Hence, the $\LO$-action on $A_\alpha''$ satisfies
$$(zx)\ast f =(f\ten\id)\overline{\alpha}(x)=\Ad(z)((f\ten\id)\widetilde{\alpha}(x))=z(x\ast f),$$
for $f\in\LO$, and $x\in A^{**}$. We emphasize that with this structure $A_\alpha''$ is not necessarily an $\LO$-submodule of $A^{**}$, rather $\mathrm{Ad}(z):A^{**}\rightarrow A_\alpha''$ is an $\LO$-complete quotient map. 

Finally, as $\widetilde{\alpha}|_{M(A)}$ is an injective $*$-homomorphism, for all $x\in M(A)$ we have
$$\norm{x}=\norm{\widetilde{\alpha}(x)}=\norm{\widetilde{\alpha}(zx)}=\norm{zx}.$$ 
It follows that $\Ad(z):M(A)\hookrightarrow A_\alpha''$ is a $G$-equivariant isometry.

\section{Amenable $W^*$-dynamical systems}

A $W^*$-dynamical system $(M,G,\alpha)$ is \textit{amenable} \cite{ADI} if there exists a projection of norm one $P:\LI\oten M\rightarrow M\cong 1\ten M$ such that $P\circ(\lambda_s\ten \alpha_s)=\alpha_s\circ P$, $s\in G$, where $\lambda$ denotes the left translation action on $\LI$. For example, $(\LI,G,\lm)$ is always amenable, and $G$ is amenable if and only if the trivial action $G\acts\{x_0\}$ is amenable, in which case $P$ becomes a left invariant mean on $\LI$. In this section we first establish a Reiter property for amenability, and then apply this result to obtain the Herz-Schur multiplier characterization from Theorem \ref{t:introW*}.

\subsection{A Reiter Property}

In this subection we establish a Reiter property for amenable $W^*$-dynamical systems, generalizing \cite[Th\'{e}or\`{e}me 3.3]{AD87} from discrete groups to arbitrary locally compact groups. We require several preparations. The first is a continuous version of \cite[Lemme 3.1]{AD87}.

Given a locally compact Hausdorff space $S$ with positive Radon measure $\mu$, and a von Neumann algebra $M$, we let 
$$K_1^+(S,Z(M)_c)=\bigg\{ g\in C_c(S,Z(M)_c^+)\mid \int_S g(s) \ d\mu(s)=1\bigg\},$$
where $C_c(S,Z(M)_c^+)$ is the space of norm continuous $Z(M)_c^+$-valued functions on $S$ with compact support. 
Let $\mc{B}_M(L^\infty(S) \overline{\otimes} M,M)$ denote the Banach space of bounded $M$-bimodule maps from $L^\infty(S) \overline{\otimes} M$ to $M$, and let $\P$ denote the convex subset of $\mc{B}_M(L^\infty(S) \overline{\otimes} M,M)$ given by the unital positive $M$-bimodule maps. Every map $P\in \mc{P}$ is automatically completely positive, so that $\norm{P}=\norm{P(1)}=1$.

Each $g\in K_1^+(S,Z(M)_c)$ gives rise to an element $P_g\in\P$ by means of the formula
$$\langle P_g(F), \omega \rangle=\int_S \langle F(s)g(s), \omega \rangle \ d\mu(s), \ \ \ F\in L^\infty(S) \overline{\otimes} M, \ \omega \in M_*.$$
The latter expression makes sense irrespective of the choice of representative of $F$ since it is equal to $\langle \tilde \omega(g),F \rangle$, viewing $g \in L^1(S,M)$. We will usually shorten the previously displayed formula by writing $P_g(F) = \int_S F(s)g(s) \, d\mu(s)$ for $F\in L^\infty(S) \overline{\otimes} M$.

Let $\P_K:=\{P_g\mid g\in K_1^+(S,Z(M)_c)\}\subseteq\P$. 

\begin{lem}\label{l:AD} Let $S$ be a locally compact Hausdorff space with positive Radon measure $\mu$ and let $M$ be a commutative von Neumann algebra. Then $\P_K$ is dense in $\P$ in the point-weak* topology of $\mc{B}(L^\infty(S) \overline{\otimes} M,M)$. \end{lem}

\begin{proof} The majority of the proof follows that of \cite[Lemme 3.1]{AD87}, but we include some details for the convenience of the reader. First, 
$$\mc{B}_M(L^\infty(S) \overline{\otimes} M,M)=((L^\infty(S) \overline{\otimes} M\ten^\pi_M M_*)^*,$$
where $\ten^\pi_M$ is the $M$-bimodule Banach space projective tensor product. By definition of the projective tensor norm together with the Radon--Nikodym theorem, every element in $(L^\infty(S) \overline{\otimes} M)\ten^\pi_M M_*$ is the equivalence class of an element of the form $F\ten \vphi$ with $F\in L^\infty(S) \overline{\otimes} M$ and $\vphi\in M_*^+$, as shown in \cite[Lemme 3.1]{AD87}.  By convexity it suffices to show that $\P$ is contained in the bipolar of $\P_K$. Let $F_0\in L^\infty(S) \overline{\otimes} M$ and $\vphi\in M_*^+$ be such that 
$$\mathrm{Re}\la P_g, F_0\ten \vphi\ra=\mathrm{Re} \ \vphi\bigg(\int_S F_0(s)g(s) \ d\mu(s)\bigg)\leq 1, \ \ \ g\in K_1^+(S,M_c).$$
If $H_0=\mathrm{Re}(F_0)$, then 
$$\vphi\bigg(\int_S H_0(s)g(s) \ d\mu(s)\bigg)\leq 1,\ \ \ g\in K_1^+(S,M_c).$$
Let $C$ denote the weak*-closure of $\{\int_S H_0(s)g(s) \, d\mu(s) \mid g\in K_1^+(S,M_c)\}$ in $M$. Given $x_1,x_2\in C$ and a projection $e\in M$, we have $x_1e+x_2(1-e)\in C$. Indeed, pick nets $(g_i),(f_j)$ in $K_1^+(S,M_c)$ such that 
$$x_1=w^*\lim_i\int_S H_0(s)g_i(s) \ d\mu(s), \ \ \ x_2=w^*\lim_j\int_S H_0(s)f_j(s) \ d\mu(s).$$
Without loss of generality, we can assume the nets $(g_i)$ and $(f_j)$ have the same index set. Since $M_c$ is weakly dense in $M$, by Kaplansky's density theorem, pick a net $(p_k)$ of positive operators in the unit ball of $M_c$ such that $p_k\rightarrow e$ strongly (and hence weak*, by boundedness). Then
$$x_1e+x_2(1-e)=w^*\lim_k \ w^*\lim_i\int_S H_0(s)(g_i(s)p_k+f_i(s)(1-p_k)) \ d\mu(s).$$
Since $g_i(1\ten p_k)+f_i(1\ten(1-p_k))\in K_1^+(S,M_c)$, combining the iterated limit into a single net, we see that $x_1e+x_2(1-e)\in C$. Then $C$ is closed under finite suprema using the Stonian structure of the spectrum of $M$, as in \cite[Lemme 3.1]{AD87}.

Now, fix a $*$-monomorphism $\rho: L^\infty(S,\mu)\rightarrow\ell^\infty(S,\mu)$, satisfying $q\circ\rho=\id_{L^\infty(S,\mu)}$, where $\ell^\infty(S,\mu)$ is the $C^*$-algebra of bounded $\mu$-measurable functions on $S$, and $q:\ell^\infty(S,\mu)\rightarrow L^\infty(S,\mu)$ is the canonical quotient map. Such a lifting exists by \cite[Corollary 2]{II}. Fix $s\in S$. Then $e_s:=\mathrm{ev}_s\circ\rho\in L^\infty(S,\mu)^*$ is a state on $L^\infty(S,\mu)$. Let $(g_i^s)$ be a net of states in $L^1(S,\mu)$ approximating $e_s$ weak*. By a further approximation using norm density of $C_c(S)$ in $L^1(S,\mu)$, we may take each $g_i^s\in C_c(S)^+$ with $\int_S g_i^s(t) \, d\mu(t) =1$.  Viewing $g_i^s \in K_1^+(S,M_c)$ in the canonical way ($M_c$ is unital), for every $F\in L^\infty(S) \overline{\otimes} M$, define a function $F_\rho: S \to M$ by
\[F_\rho(s) = w^*\lim_i \int_S g_i^s(t) F(t) \, d\mu(t) = w^*\lim_i P_{g_i^s}(F).\] To see that this definition makes sense regardless of representative of $F$, note that for $\omega \in M_*$, $\int_S \langle g_i^s(t)F(t), \omega \rangle \, d\mu(t) = \langle g_i^s, (\mathrm{id} \otimes \omega)(F) \rangle$, where $(\mathrm{id} \otimes \omega)(F)$ is the element in $L^\infty(S,\mu)$ defined $\langle (\mathrm{id} \otimes \omega)(F),g \rangle = \langle F, g \otimes \omega \rangle$ for $g \in L^1(S,\mu)$.

We claim that $F_\rho$ represents $F$. Indeed, to check measurability of $\varphi_g: s \mapsto \langle F_\rho(s),g(s) \rangle$ for all $g \in L^1(S,M_*)$, first take $g$ to be a simple tensor in $L^1(S,M_*) = L^1(S) \otimes^\pi M_*$. In this case, $\varphi_g$ is a product of measurable functions, hence measurable. This implies the claim for all simple functions $g \in L^1(S,M_*)$ since these are sums of simple tensors. The claim for general $g$ follows from this since a pointwise a.e.--limit of measurable functions is measurable. The formula $\langle F,g \rangle = \int_S \langle F(s), g(s) \rangle \, d\mu(s)$ for $g \in L^1(S,M_*)$ is then readily checked for simple tensors $g$ and improved to general $g$ using the observation that $F_\rho$ is bounded.

Since $(H_0)_\rho(s) \in C$, we have $(H_0)_\rho(s)^+ = (H_0)_\rho(s) \vee 0 \in C$. Define $m=\sup_{s\in S}(H_0)_\rho(s)^+\in M$. Then by normality of $\vphi$, we have $\vphi(m)\leq 1$. Since $(H_0)_\rho(s) \leq m$ in $M$ for all $s$, it follows that $H_0 \leq 1 \otimes m$ in $L^\infty(S) \overline{\otimes} M$. Indeed, if $g \in L^1(S,M_*)^+$ is a positive normal functional on $L^\infty(S) \overline{\otimes} M$, then
\[ \langle H_0,g \rangle = \int_S \langle (H_0)_\rho(s),g(s) \rangle \, d\mu(s) \leq \int_S \langle m, g(s) \rangle \, d \mu(s) = \langle 1 \otimes m,g \rangle.\]

Thus, for every $P\in\P$ we have
$$P(H_0)\leq P(1\ten m)=m P(1)=m,$$ 
so that  
$$\mathrm{Re}\la P,F_0\ten\vphi\ra=\vphi(P(H_0))\leq\vphi(m)\leq 1.$$
Hence, $P$ belongs to the bipolar of $\P_K$.
\end{proof}

\begin{remark} In the special case where $M=L^\infty(X,\nu)$ and $(X,\nu)$ and $(S,\mu)$ are both $\sigma$-finite, the conclusion of Lemma \ref{l:AD} follows from \cite[Lemma 1.2.6]{ADR}.
\end{remark}

Similar to \cite{AD87}, we consider the following two locally convex topologies on the Bochner space $L^1(S,M)$, where $S$ and $M$ are as in Lemma \ref{l:AD}. The first, denoted $\tau_n$, is generated by the family of semi-norms $\{p_\om\mid\om\in M_*^+\}$, where
$$p_\om(g)=\la \om, \int_S |g(s)| \ d\mu(s)\ra = \int_S \la |g(s)|,\om \ra \ d\mu(s) .$$
This is indeed well-defined since $s \mapsto |g(s)|$ is Bochner integrable whenever $g$ is.

The second, denoted $\tau_F$, is generated by the family of semi-norms
$$\{p_{F,\om}\mid F\in L^\infty(S) \overline{\otimes} M, \ \om\in M_*^+\}, \ \ \ \textnormal{where} \ \ \ p_{F,\om}(g)=\bigg|\int_S\la g(s)F(s),\om\ra \ d\mu(s)\bigg|.$$
To see that this is well-defined, define $\tilde{\om}(g): S \to M_*$ by $\la x, \tilde{\om}(g)(s) \ra = \la g(s)x,\om \ra$ for $x \in M$. Then by Lemma \ref{l:1} $\tilde{\om}(g) \in L^1(S,M_*)$. A routine argument then shows that $s \mapsto \la F(s),\tilde{\om}(g)(s) \ra = \la g(s)F(s),\om \ra$ is measurable, and integrability of this function is easy to check.

Since $p_{F,\om}(g)\leq \norm{F}p_{\om}(g)$, it follows that $\tau_n$ is stronger than $\tau_F$.



\begin{lem}\label{l:top} Let $V$ be a convex subset of $L^1(S,M)$ such that every function in $V$ is supported on a $\sigma$-finite subset. Then $\overline{V}^{\tau_F}=\overline{V}^{\tau_n}$. 
\end{lem}

\begin{proof}Since $\tau_n$ is stronger than $\tau_F$, it suffices to show that $\overline{V}^{\tau_F}\subseteq\overline{V}^{\tau_n}$. Let $(g_i)$ be a net in $V$ converging to zero with respect to $\tau_F$.
Then, by definition of $\tau_F$, $\tilde{\omega}(g_i) \to 0$ weakly in $L^1(S,M_*)$ for all $\omega \in M_*^+$. By Mazur's theorem, there exists a net $(g_{K,\ep})$ in $V$ indexed by finite subsets $K$ of $M_*^+$ and $\ep>0$ such that 
$$\norm{\tilde{\omega}(g_{K,\ep})}_{L^1(S,M_*)}<\ep, \ \ \ \om\in K.$$
For (an a.e.-representative of) $g \in L^1(S,M)$ and $s \in S$, let $g(s) = u_s|g(s)|$ be the polar decomposition in $M$. Then, since
\[ \langle \omega, |g(s)| \rangle = | \langle \tilde{\omega}(g)(s), u_s^* \rangle | \leq \sup\{|\langle \tilde{\omega}(g)(s),x \rangle| : x \in M_{\norm{\cdot}\leq 1}\} = \|\tilde \om (g)(s)\|_{M_*} \]
for all $\omega \in M_*^+$ and $s \in S$, we have
\[p_\omega(g) = \int_S \langle \omega,u_s^*g(s) \rangle \, d\mu(s) \leq \int_S \|\tilde{\omega}(g)(s)\|_{M_*} \, d\mu(s) =  \norm{\tilde{\omega}(g)}_{L^1(S,M_*)}\] for all $g \in L^1(S,M)$. It follows that $g_{K,\ep} \to 0$ with respect to $\tau_n$.
\end{proof}

The next lemma will be used to upgrade pointwise asymptotic $G$-invariance in Reiter's property to uniform asymptotic $G$-invariance on compacta. This is a generalization of the equivalence of the classical finite and compact Reiter's properties. Our proof generally follows that of \cite[Proposition 6.10]{Pier}.

First, we record a useful, simple lemma, the proof of which is omitted.

\begin{lem} \label{l:bdd w*=ucc}
	Suppose $X$ is a Banach space and $(\varphi_t)$ is a bounded net in $X^*$. Then $\varphi_t \to \varphi$ weak* in $X^*$ if and only if $\varphi_t(x) \to \varphi(x)$ uniformly on compact subsets of $X$.
\end{lem}


\begin{lem}\label{l:topamenw*} Let $(M,G,\alpha)$ be a commutative $W^*$-dynamical system. The following conditions are equivalent:
\begin{enumerate}
\item There exists a net $(g_i)$ in $K_1^+(G,M_c)$ such that
$$w^*\lim_i\int_G|g_i(s)-(\lm_t\ten\alpha_t)(g_i)(s)| \ ds=0, \ \ \ \text{for all } t\in G.$$
\item There exists a net $(g_i)$ in $K_1^+(G,M_c)$ such that 
$$w^*\lim_i\int_G|g_i(s)-(\lm_t\ten\alpha_t)(g_i)(s)| \ ds=0, \ \ \ 
\text{uniformly on compact subsets of }G.$$
\end{enumerate}
\end{lem}

\begin{proof} Since $(2)$ clearly implies $(1)$, we only need to show $(1)$ implies $(2)$. In preparation, note that $C_c(G,M_c)$ is a left module over the algebra $\mathcal M_c(G)$ of compactly supported Radon measures on $G$ via the action
	\[ \mu \star g(s) = \int_G (\lambda_t \otimes \alpha_t)(g)(s) \, d \mu(t) = \int_G \alpha_t(g(t^{-1}s)) \, d\mu(t), \]
for $\mu \in \mathcal M_c(G)$, $g \in C_c(G,M_c)$, and $s \in G$.
(In fact, this action extends to give the injective Banach space tensor product $L^1(G) \otimes^{\epsilon} M_c$ a left Banach $\mathcal M(G)$-module action, but we will not need this.)
Note also that for $g \in C_c(G,M_c)$, $\mu \in \mathcal M_c(G)$, and $\omega \in M_*^+$,
\begin{align*}
\la \om, \int_G |\mu \star g(s)| \, ds \ra
& = \int_G \la \om, \left| \int_G \alpha_t(g(t^{-1}s)) \, d\mu(t) \right| \ra \, ds \\
& \leq \int_G \la \om, \int_G |\alpha_t(g(t^{-1}s))| \, d|\mu|(t) \ra \, ds \\
& = \int_G \int_G \la \om, |\alpha_t(g(t^{-1}s)) | \ra \, ds \, d|\mu|(t) \\
& = \int_G \int_G \la \om, \alpha_t(|g(s)|) \ra \, ds \, d|\mu|(t) \\
& = \int_G \int_G \la (\alpha_t)_*(\om), |g(s)| \ra \, ds \, d|\mu|(t) \\
& = \int_G \la (\alpha_t)_*(\om), \int_G |g(s)| \, ds \ra \, d|\mu|(t)
\numberthis \label{e:prepineq1}
\\
& \leq \|\om\| \|\mu\| \left\| \int_G |g(s)| \, ds \right\|.
\numberthis \label{e:prepineq2}
\end{align*}
(Note that commutativity of $M_c$ is used in the first inequality.)


Let $(g_i)$ be a net as in $(1)$, and fix $f\in C_c(G)^+$ with $\int_G f(s) \, ds = 1$. We will show that $(f \star g_i)$ satisfies $(2)$. It is immediate that $f \star g_i \in C_c(G,M_c^+)$, and the equality $\int_G f \star g_i(s) \, ds = 1$ is straightforward to check. 


The remainder of the proof closely follows that of \cite[Proposition 6.10]{Pier}. Let $C\subseteq G$ be compact, $\om \in (M_*)_{\norm{\cdot}=1}^+$, and $\ep>0$. Put $C_1=C\cup\{e\}$, and $\delta=\ep/6$. There exists a neighborhood $U$ of $e$ such that $\norm{\lm_{y}f-f}_{\LO}<\delta$, whenever $y\in U$ (by \cite[20.4]{HR}). Since $C_1$ is compact, there exists a compact neighbourhood $V$ of $e$ such that $t^{-1}Vt\subseteq U$ for every $t\in C_1$ (by \cite[4.9]{HR}). Hence for every $r\in V$ and $t\in C_1$, $\norm{\lm_{t^{-1}rt}f-f}_{\LO}<\delta$. Let $h=|V|^{-1}\chi_V$. Then for $t\in C_1$ we have
\begin{align*}\norm{h\ast(\lm_{t}f)-\lm_{t}f}_{\LO}&=\int_G\bigg|\int_Gh(r)f(t^{-1}r^{-1}s) \ dr-\int_Gh(r)f(t^{-1}s) \ dr\bigg| \ ds\\
&\leq\int_G h(r)\bigg(\int_G|f(t^{-1}r^{-1}ts)-f(s)| \ ds\bigg)\ dr\\
&=\int_V h(r)\norm{\lm_{t^{-1}rt}f-f}_{\LO} \ dr<\delta. \numberthis \label{e:prepineq3}
\end{align*}
Let $C' = C_1 \supp(f)$. Then $C'$ is compact, and for every $t \in C_1$,
\begin{equation}\label{e:d1} \int_{G \setminus C'} f(t^{-1}s) \, ds = 0. \end{equation}
Pick a compact neighbourhood $W$ of $e$ such that $\norm{h\ast\delta_t-h}_{\LO} = \norm{\lambda_{t^{-1}}h^* - h^*}_{\LO}<\delta$ for every $t\in W$ (by \cite[20.4]{HR} again). Then there is an open neighbourhood $W'$ of $e$ for which $W'W'^{-1}\subseteq W$. As $C'$ is compact, there are $c_1,...,c_m\in C'$ such that $C'\subseteq\cup_{i=1}^m W'c_i$. Since each $W'c_i$ satisfies $W'c_i(W'c_i)^{-1}\subseteq W$, there exists a finite partition $\{B_j\mid j=1,...,n\}$ for $C'$ consisting of non-empty Borel sets such that $B_jB_j^{-1}\subseteq W$ for all $j$. For every $j=1,...,n$ choose $b_j\in B_j$. Then for all $t\in B_j$
\begin{equation}\label{e:d2}\norm{h\ast\delta_t-h\ast\delta_{b_j}}_{\LO}=\norm{h\ast\delta_{tb_{j}^{-1}}-h}_{\LO}<\delta.\end{equation}

Now, using norm-compactness of $\{(\alpha_t)_*(\om)\mid t\in V\}$ and Lemma \ref{l:bdd w*=ucc} again, condition (1) implies that for some index value $i_0$,
$$\la(\alpha_t)_*(\om),\int_G |g_i(s)-(\delta_{b_j}\star g_i)(s)| \ ds\ra<\delta,$$
for every $i \geq i_0$, $j=1,...,n$ and $t\in V$. To simplify notation for the following calculations, fix $g=g_i$ for $i \geq i_0$. Then for $j=1,...,n$, it follows by the inequality (\ref{e:prepineq1}) above that
\begin{align*}
\la\om, \int_G |h\star\delta_{b_j}\star g(s)-h\star g(s)| \ ds\ra
&\leq \int_G h(t) \, \la(\alpha_t)_*(\om), \int_G | (\delta_{b_j} \star g - g)(s)| \, ds \ra \ dt\\
&< \int_G h(t) \delta \, dt\\
& = \delta. \numberthis \label{e:d3}
\end{align*}

Now, for $t\in C_1$, we have by inequalities (\ref{e:prepineq2}) and (\ref{e:prepineq3}) 
\begin{align*}\int_G\la\om,|(\lm_t f) \star g(s) - h \star (\lm_t f) \star g(s)| \ra \, ds
&= \int_G \la \om, |(\lm_t f - h * (\lm_tf))\star g(s)| \ra \, ds \\
& \leq \| \om\| \|\lm_t f - h *(\lm_t f) \| \left\| \int_G |g(s)| \, ds \right\| \\
& < \delta. \numberthis \label{e:d4}
\end{align*}
Also, if $f'=\lm_{t}f$ for $t \in C_1$, then $f'$ is a state and applying (\ref{e:d1}), (\ref{e:d2}), (\ref{e:d3}), and (\ref{e:prepineq2}), we see that
\begin{align*}
&\int_G\la\om, |(h\star f'\star g-h\star g)(s)|\ra \ ds\\
&\leq\int_G\int_G f'(r)\la\om,|(h\star\delta_r\star g-h\star g)(s)|\ra \ dr \ ds\\
&=\int_G \int_{C'} f'(r)\la\om, |(h\star\delta_r\star g-h\star g)(s)|\ra \ dr \ ds\\
&\leq \int_G\bigg(\sum_{j=1}^n\int_{B_j}f'(r)\la\om,|((h\ast\delta_r-h\ast\delta_{b_j})\star g)(s)|\ra\ dr \ +\\
& \ \ \ \ \ \ \ \sum_{j=1}^n\int_{B_j}f'(r)\la\om,|(h\star\delta_{b_j}\star g-h\star g)(s)|\ra \ dr\bigg) \ ds\\
&=\sum_{j=1}^n\int_{B_j}f'(r)\int_G\la\om,|((h\ast\delta_r-h\ast\delta_{b_j})\star g)(s)|\ra \ ds \ dr \ +\\
& \ \ \ \ \ \ \ \sum_{j=1}^n\int_{B_j}f'(r)\int_G\la\om,|(h\star\delta_{b_j}\star g-h\star g)(s)|\ra \ ds \ dr\\
&<2\delta. \numberthis \label{e:d5}
\end{align*}
Finally, let $t\in C$. Since $t\in C_1$ and $e\in C_1$, by (\ref{e:d4}) and (\ref{e:d5}) we have
\begin{align*}&\int_G\la\om,|(\delta_{t}\star(f\star g)-f\star g)(s)|\ra \ ds\\
&\leq \int_G\la\om,|(f'\star g-h\star(f'\star g))(s)|\ra \ ds \ +\\
&\ \ \ \ \ \ \ \int_G\la\om, |(h\star(f'\star g))-h\star g)(s)|\ra \ ds + \int_G\la\om,|(h\star g-h\star f\star g)(s)|\ra \ ds \ +\\
&\ \ \ \ \ \ \ \int_G\la\om,|(h\star f\star g-f\star g)(s)|\ra \ ds\\
&<\delta+2\delta+2\delta+\delta=\ep.
\end{align*}
It follows that the net $(f\star g_i)$ satisfies
\begin{equation*}\label{e:ucconv}w^*\lim_i\int_G|f\star g_i(s)-(\lm_t\ten\alpha_t)(f\star g_i)(s)| \ ds=0,\end{equation*}
uniformly for $t$ in compact subsets of $G$.
\end{proof}

We are now in position to generalize \cite[Th\'{e}or\`{e}me 3.3]{AD87} to locally compact groups. The equivalences in the next theorem were independently obtained for exact locally compact groups using different techniques by Buss--Echterhoff--Willett in the recent work \cite{BEW}. 


\begin{thm}\label{t:Reiter} Let $(M,G,\alpha)$ be a $W^*$-dynamical system. The following conditions are equivalent:
\begin{enumerate}
\item There exists a net $(h_i)$ of positive type functions in $C_c(G,Z(M)_c)$ such that
\begin{enumerate}
\item $h_i(e)=1$ for all $i$;
\item $\lim_i h_i(t)=1$ weak*, uniformly on compact subsets.
\end{enumerate}
\item There exists a net $(\xi_i)$ in $C_c(G,Z(M)_c)$ such that
\begin{enumerate}
\item $\la\xi_i,\xi_i\ra=1$ for all $i$;
\item $\la\xi_i,(\lm_t\ten\alpha_t)\xi_i\ra\rightarrow 1$ weak*, uniformly on compact subsets.
\end{enumerate}
\item There exists a net $(g_i)$ in $K_1^+(G,Z(M)_c)$ such that
$\int_G |(\lm_t\ten\alpha_t)g_i(s)-g_i(s)| \ ds\rightarrow 0$ weak*, uniformly on compact subsets.
\item There exists a $G$-equivariant projection of norm one from $\LI\oten M$ onto $M$.
\item There exists a $G$-equivariant projection of norm one from $\LI\oten Z(M)$ onto $Z(M)$. 
\end{enumerate}
\end{thm}

\begin{proof} $(2) \Rightarrow (1)$ is obvious by taking $h_i(t) = \la\xi_i,(\lambda_t \otimes \alpha_t)\xi_i \ra$ (noting that the compact support of $\xi_i$ implies the range of $h_i$ indeed lies in the norm closed subalgebra $Z(M)_c$).

$(1)\Rightarrow(2)$: By \cite[Proposition 2.5]{AD87} there exists a net $(\xi_i)$ in $L^2(G,Z(M)_c)$ satisfying properties (2)(a) and (2)(b). By norm density of $C_c(G,Z(M)_c)$ in $L^2(G,Z(M)_c)$, a further approximation yields a net $(\eta_i)$ in $C_c(G,Z(M)_c)$ satisfying $\la\eta_i,\eta_i\ra\leq 1$ for all $i$, $\la\eta_i,\eta_i\ra\rightarrow 1$ in norm, and property (2)(b). By continuity of the continuous functional calculus on $[0,1]$, it follows that $\sqrt{\la\eta_i,\eta_i\ra}\rightarrow 1$ in norm. Pick $i_0$ such that $\norm{1-\sqrt{\la\eta_i,\eta_i\ra}}<1/2$ for all $i\geq i_0$, and redefine $\xi_i(s):=\la\eta_i,\eta_i\ra^{-1/2}\eta_i(s)$, for $i\geq i_0$. Then by commutativity
$$\la\xi_i,\xi_i\ra=\la\eta_i,\eta_i\ra^{-1}\la\eta_i,\eta_i\ra=1,  \ \ \ i\geq i_0,$$
and
$$\la\xi_i,(\lm_t\ten\alpha_t)\xi_i\ra=\la\eta_i,\eta_i\ra^{-1/2}\alpha_t(\la\eta_i,\eta_i\ra^{-1/2})\la\eta_i,(\lm_t\ten\alpha_t)\eta_i\ra\rightarrow 1$$
weak*. Thus, the $(\xi_i)_{i\geq i_0}$ satisfies (2).

$(2)\Leftrightarrow(3)$ follows more or less immediately from \cite[Lemme 3.2]{AD87} applied to the commutative $C^*$-dynamical system $(Z(M)_c,G,\alpha)$.

$(3)\Rightarrow(4)$: Suppose there exists a net $(g_i)$ in $K_1^+(G,Z(M)_c)$ satisfying condition 3 above.
By the properties of $(g_i)$, each $P_{g_i}$ is a positive unital $M$-bimodule map. Passing to a subnet we may assume that $(P_{g_i})$ converges weak* to some $P$ in $\mc{B}(\LI\oten M,M)$, which is necessarily a projection of norm one.

Fix $t\in G$, $F\in(\LI\oten M)^+$ and $\om\in M_*^+$. Choose a representation for $F$ with values in $M_+$. Then
\begin{align*}\la P(\lambda_t\ten \alpha_t(F)),\om\ra&=\lim_i\int_G\la g_i(s)(\lambda_t\ten\alpha_t)(F)(s),\om\ra \ ds\\
&=\lim_i\int_G\la \alpha_t(\alpha_{t^{-1}}(g_i(s))F(t^{-1}s)),\om\ra \ ds\\
&=\lim_i\int_G\la \alpha_{t^{-1}}(g_i(s))F(t^{-1}s),(\alpha_t)_*(\om)\ra \ ds\\
&=\lim_i\int_G\la \alpha_{t^{-1}}(g_i(ts))F(s),(\alpha_t)_*(\om)\ra \ ds\\
&=\lim_i\int_G\la ((\lm_{t^{-1}}\ten\alpha_{t^{-1}})g_i)(s)F(s),(\alpha_t)_*(\om)\ra \ ds.\end{align*}
Since $(\lm_t\ten\alpha_t)g_i(s)-g_i(s)\in Z(M)_c$ is self-adjoint for each $s\in G$, we have
\begin{align*}&\la ((\lm_{t^{-1}}\ten\alpha_{t^{-1}})g_i-g_i)(s)F(s),(\alpha_t)_*(\om)\ra\\
&=\la \sqrt{F(s)}((\lm_{t^{-1}}\ten\alpha_{t^{-1}})g_i-g_i)(s)\sqrt{F(s)},(\alpha_t)_*(\om)\ra\\
&\leq\la \sqrt{F(s)}|((\lm_{t^{-1}}\ten\alpha_{t^{-1}})g_i-g_i)(s)|\sqrt{F(s)},(\alpha_t)_*(\om)\ra\\
&=\la |((\lm_{t^{-1}}\ten\alpha_{t^{-1}})g_i-g_i)(s)|F(s),(\alpha_t)_*(\om)\ra\\
&\leq \norm{F}\la|((\lm_{t^{-1}}\ten\alpha_{t^{-1}})g_i-g_i)(s)|,(\alpha_t)_*(\om)\ra,
\end{align*}
for every $s,t\in G$. Property $3$ of $(g_i)$ then implies that
$$\la P(\lambda_t\ten \alpha_t(F)),\om\ra=\lim_i\int_G\la g_i(s)F(s),(\alpha_t)_*(\om)\ra=\la\alpha_t(P(F)),\om\ra,$$
which yields (4).

$(4)\Rightarrow(5)$ is obvious by restriction, using the $1\ten M-M$-bimodule property of projections of norm one $\LI\oten M\rightarrow M$.

$(5) \Rightarrow (3)$ is all that remains. Let $P:\LI\oten Z(M)\rightarrow Z(M)$ be a $G$-equivariant projection of norm one. By Lemma \ref{l:AD} applied to the commutative von Neumann algebra $ Z(M)$, $P$ lies in the point-weak* closure of $\P_K$. Hence, there is a net $(g_i)$ of functions in $K_1^+(G, Z(M)_c)$ satisfying
$$P(F)=w^* \lim_iP_{g_i}(F)= w^* \lim_i\int_G g_i(s)F(s) \ ds, \ \ \ F\in L^\infty(G) \overline{\otimes} Z(M).$$
The $G$-equivariance of $P$ implies that
$$\lim_i\int_G \la g_i(s)(\lambda_t\ten\alpha_t)(F)(s),\om\ra  ds=\lim_i\int_G\la g_i(s)F(s),(\alpha_t)_*(\om)\ra \ ds $$
for all $F\in L^\infty(G) \overline{\otimes} Z(M)$, $\om\in  Z(M)_*$ and $t\in G$. But
$$\int_G \la g_i(s)(\lambda_t\ten\alpha_t)(F)(s),\om\ra \ ds=\int_G \la (\lambda_{t^{-1}}\ten\alpha_{t^{-1}})(g_i)(s)F(s),(\alpha_t)_*(\om)\ra \ ds,$$
as shown above, so it follows that $((\lambda_{t}\ten\alpha_{t})(g_i)-g_i)\rightarrow 0$ with respect to $\tau_F$ (on $L^1(G,Z(M))$) for all $t\in G$. Just as in \cite[pg.\ 307]{AD87}, one can use Lemma \ref{l:top} applied to $V=K_1^+(G,Z(M)_c)$ and an argument involving direct sums of copies of $L^1(G,Z(M))$ to show the existence of a net $(g_j)$ in $K_1^+(G, Z(M)_c)$ such that $((\lambda_{t}\ten\alpha_{t})(g_j)-g_j)\rightarrow 0$ with respect to $\tau_n$ for all  $t\in G$, which implies a pointwise version of (3). Property (3) then follows from Lemma \ref{l:topamenw*}.
\end{proof}

\begin{remark}\label{r:Reiter} The corresponding result remains true if (1)(a) is replaced with $h_i(e)\leq 1$ for all $i$, (2)(a) is replaced with $\la\xi_i,\xi_i\ra\leq 1$ for all $i$ and the net $(g_i)$ in (3) satisfies $g_i\in C_c(G,Z(M)_c^+)$, $\int_G g(s) \ ds\leq 1$ instead of $g_i\in K^+_1(G,Z(M)_c)$. This follows by the same techniques, using a ``contractive'' version of Lemma \ref{l:AD}.
\end{remark}

As a corollary to Theorem \ref{t:Reiter} (and its proof), we obtain a different proof of the fact that a $W^*$-dynamical system $(M,G,\alpha)$ over an arbitrary locally compact group $G$ is amenable if and only if the restricted action $( Z(M),G,\alpha)$ is amenable \cite[Corollaire 3.6]{ADII}.

For actions of second countable locally compact groups $G$ on standard Borel spaces $(X,\mu)$ with a quasi-invariant measure $\mu$, amenability of $(L^\infty(X,\mu),G,\alpha)$ implies that $\pi_X$ is weakly contained in $\lm$ \cite[Corollary 3.2.2]{AD03}, where $\pi_X$ is the associated unitary representation of $G$ on $L^2(X,\mu)$. As a corollary to Theorem \ref{t:Reiter}, we obtain a generalization of this fact to arbitrary $(M,G,\alpha)$.

\begin{cor}\label{c:Reiter} Let $(\pi,u)$ be a normal covariant representation of a $W^*$-dynamical system $(M,G,\alpha)$. If $(M,G,\alpha)$ is amenable then $u$ is weakly contained in $\lm$.
\end{cor}


\begin{proof}
Let $(\xi_i)$ be as in Theorem \ref{t:Reiter} (2). Fix $v \in H$, and define $\eta_i: G \to H$ by $\eta_i(t) = u(t^{-1})\xi_i(t)v$. Then $\eta_i \in L^2(G,H)$, and a calculation similar to one in the proof of \cite[Theorem 5.3]{AD} gives \[\la \eta_i, \lm_s\eta_i \ra = \la v, \la \xi_i,(\alpha_s \otimes \lm_s)\xi_i \ra u_s v \ra \to \la v, u_s v \ra\] uniformly on compact subsets of $G$.
\end{proof}

\subsection{Herz-Schur Multipliers}

The theory of Herz-Schur multipliers has recently been generalized to the setting of dynamical systems \cite{BC,BC2,MTT,MSTT}. In this subsection we build on this work by providing an explicit representation of Herz-Schur multipliers arising from compactly supported positive type functions for arbitrary $(M,G,\alpha)$, along with a multiplier characterization of amenability. We begin with preliminaries on Hilbert $C^*$-modules associated to dynamical systems.

Let $(A,G,\alpha)$ be a $C^*$-dynamical system. We let $L^2(G,A)$ be the right Hilbert $A$-module given by the completion of $C_c(G,A)$ under $\norm{\xi}=\norm{\la\xi,\xi\ra}_A^{1/2}$, where
$$\la\xi,\zeta\ra=\int_G\xi(s)^*\zeta(s) \ ds, \ \ \ \xi\cdot a(s)=\xi(s)a, \ \ \ \xi,\zeta\in C_c(G,A), \ a\in A.$$
To simplify notation we let $\widetilde{\alpha}_t\in \mc{B}(L^2(G,A))$ denote the isometry
$$\wtal_t\xi(s):=(\lm_t\ten\alpha_t)\xi(s)=\alpha_t(\xi(t^{-1}s)), \ \ \ \xi\in C_c(G,A).$$
By left invariance of the Haar measure and continuity of the action it follows that 
$$\la\wtal_t\xi,\wtal_t\zeta\ra=\alpha_t(\la\xi,\zeta\ra), \ \ \ \xi,\zeta\in L^2(G,A), \ t\in G.$$
We assume throughout that $A\subseteq\BH$ non-degenerately. Then $\alpha:A\rightarrow C_b(G,A)\subseteq\mc{B}(L^2(G,H))$ is a strict $*$-homomorphism, and viewing $L^2(G,H)$ as a right Hilbert $C^*$-module over $\bC$, we may form the interior tensor product $L^2(G,A)\ten_\alpha L^2(G,H)$ \cite[Proposition 4.5]{Lance}. This becomes a Hilbert space with inner product given on simple tensors by
$$\la\xi_1\ten_\alpha\eta_1,\xi_2\ten_\alpha\eta_2\ra=\la\eta_1,\alpha(\la\xi_1,\xi_2\ra)\eta_2\ra.$$
Letting $\pi:A\ni a\mapsto 1\ten a\in\mc{B}(L^2(G,H))$, we also implicitly use the interior tensor product $L^2(G,A)\ten_\pi L^2(G,H)$, which is a Hilbert space under the inner product
$$\la\xi_1\ten_\pi\eta_1,\xi_2\ten_\pi\eta_2\ra=\la\eta_1,(1\ten\la\xi_1,\xi_2\ra)\eta_2\ra.$$
The map
$$L^2(G,A)\ten_\pi L^2(G,H)\ni \xi\ten_\pi\eta\mapsto \xi\cdot\eta \in L^2(G\times G,H)$$
extends to a unitary operator, where
$$\xi\cdot\eta(s,t)=\xi(s)\eta(t), \ \ \ s,t\in G.$$
Indeed, for any $\xi_1,...,\xi_n\in C_c(G,A)$ and $\eta_1,...,\eta_n\in C_c(G,H)$, 
\begin{align*}\norm{\sum_{i=1}^n\xi_i\cdot\eta_i}_{L^2(G\times G,H)}^2&=\iint\sum_{i,j=1}^n\la\xi_i(s)\eta_i(t),\xi_j(s)\eta_j(t)\ra_H \ ds \ dt\\
&=\sum_{i,j=1}^n\iint\la\eta_i(t),\xi_i(s)^*\xi_j(s)\eta_j(t)\ra_H \ ds \ dt\\
&=\sum_{i,j=1}^n\int\la\eta_i(t),\la\xi_i,\xi_j\ra\eta_j(t)\ra_H \ dt\\
&=\sum_{i,j=1}^n\la\eta_i,(1\ten\la\xi_i,\xi_j\ra)\eta_j\ra_{L^2(G,H)}\\
&=\sum_{i,j=1}^n\la\xi_i\ten_\pi\eta_i,\xi_j\ten_\pi\eta_j\ra\\
&=\norm{\sum_{i=1}^n\xi_i\ten_\pi\eta_i}^2. \numberthis \label{e:pi}
\end{align*}
The map is therefore an isometry. That it also has dense range follows from non-degeneracy of $A\subseteq\BH$ using a bai for $A$.
 
We let $W_\alpha:L^2(G,A)\ten_\alpha L^2(G,H)\rightarrow L^2(G\times G,H)$ be the map determined by 
$$W_\alpha(\xi\ten \eta)(s,t)=(\wtal_{t^{-1}}\xi)(s)\eta(t)=\alpha_{t^{-1}}(\xi(ts))\eta(t), \ \ \ \xi\in C_c(G,A), \ \eta\in C_c(G,H).$$
Since
\begin{align*}W_\alpha(\xi\cdot a\ten\eta)(s,t)&=\alpha_{t^{-1}}(\xi(ts)a)\eta(t)=\alpha_{t^{-1}}(\xi(ts))\alpha_{t^{-1}}(a)\eta(t)\\
&=\alpha_{t^{-1}}(\xi(ts))(\alpha(a)\eta)(t)\\
&=W_\alpha(\xi\ten\alpha(a)\eta)(s,t),
\end{align*}
it follows that $W_\alpha$ induces a unitary $W_\alpha:L^2(G,A)\ten_\alpha L^2(G,H)\rightarrow L^2(G\times G,H)$, since
\begin{align*}\norm{W_\alpha\bigg(\sum_{i=1}^n\xi_i\ten_\alpha\eta_i\bigg)}^2&=\iint\norm{\sum_{i=1}^nW_\alpha(\xi_i\ten_\alpha\eta_i)(s,t)}^2_H \ ds \ dt\\
&=\iint\norm{\sum_{i=1}^n\wtal_{t^{-1}}\xi_i(s)\eta_i(t)}^2_H \ ds \ dt\\
&=\iint\sum_{i,j=1}^n\la\wtal_{t^{-1}}\xi_i(s)\eta_i(t),\wtal_{t^{-1}}\xi_j(s)\eta_j(t)\ra_H \ ds \ dt\\
&=\iint\sum_{i,j=1}^n\la\eta_i(t),\wtal_{t^{-1}}\xi_i(s)^*\wtal_{t^{-1}}\xi_j(s)\eta_j(t)\ra_H \ ds \ dt\\
&=\int\sum_{i,j=1}^n\la\eta_i(t),\la\wtal_{t^{-1}}\xi_i,\wtal_{t^{-1}}\xi_j\ra\eta_j(t)\ra_H \ dt\\
&=\int\sum_{i,j=1}^n\la\eta_i(t),\alpha_{t^{-1}}(\la\xi_i,\xi_j\ra)\eta_j(t)\ra_H \ dt\\
&=\sum_{i,j=1}^n\la\eta_i,\alpha(\la\xi_i,\xi_j\ra)\eta_j\ra_{L^2(G,H)}\\
&=\sum_{i,j=1}^n\la\xi_i\ten_\alpha\eta_i,\xi_j\ten_\alpha\eta_j\ra\\
&=\norm{\sum_{i=1}^n\xi_i\ten_\alpha\eta_i}^2.
\end{align*}
This fact was observed for discrete dynamical systems in \cite[Lemma 4.9]{BC}. By covariance of $(\alpha, \lm\ten 1)$, one easily sees that $\wtal_t\ten(\lm_t\ten1)$ induces an invertible map on $L^2(G,A)\ten_\alpha L^2(G,H)$, and the standard argument shows that
\begin{equation}\label{e:fund}W_\alpha^*(1\ten(\lm_t\ten 1))W_\alpha=\wtal_t\ten(\lm_t\ten1), \ \ \ t\in G.\end{equation}
Also, whenever $a\in A$ commutes with the range of $\xi\in L^2(G,A)$, in particular, when $\xi\in L^2(G,Z(A))$, we have
\begin{align*}W_\alpha(\xi\ten_\alpha\alpha(a)\eta)(s,t)&=W_\alpha(\xi\cdot a\ten\eta)(s,t)\\
&=\alpha_{t^{-1}}(\xi\cdot a(ts))\eta(t)\\
&=\alpha_{t^{-1}}(a\xi(ts))\eta(t)\\
&=\alpha_{t^{-1}}(a)\alpha_{t^{-1}}(\xi(ts))\eta(t)\\
&=(1\ten\alpha(a))(W_\alpha(\xi\ten_\alpha\eta))(s,t).
\end{align*}
Thus, 
\begin{equation}\label{e:center} W_\alpha(\xi\ten_\alpha\alpha(a)\eta)=(1\ten\alpha(a))W_\alpha(\xi\ten_\alpha\eta).\end{equation}

When $(A,G,\alpha)=(\bC,G,\mathrm{trivial})$, $W_\alpha$ is simply the fundamental unitary of the quantum group $VN(G)$.

The following are special cases of \cite[Definitions 3.1,3.3]{MTT} when $F$ is assumed bounded and continuous.

\begin{defn}\cite[Definitions 3.1,3.3]{MTT} Let $(A,G,\alpha)$ be a $C^*$-dynamical system. A bounded continuous function $F:G\rightarrow\mc{CB}(A)$ is:
\begin{enumerate}
\item a (completely positive) \textit{Herz-Schur $(A,G,\alpha)$-multiplier} if the map 
$$\Theta(F)(\alpha\times\lm)(f)=(\alpha\times\lm)(F\cdot f), \ \ \ f\in C_c(G,A),$$
extends to a completely (positive) bounded map on $G\ltimes A$, where $F\cdot f(s)=F(s)(f(s))$, $s\in G$.
\item a (completely positive) \textit{Herz-Schur multiplier} if the map 
$$\Theta(F)(\alpha(a)(\lm_s\ten 1))=\alpha(F(s)(a))(\lm_s\ten 1), \ \ \ a\in A, \ s\in G,$$
extends to a normal completely (positive) bounded map on $(G\ltimes A)''$ (the weak*-closure of $G \ltimes A$ in $B(L^2(G,H))$).
\end{enumerate}
\end{defn}

By \cite[Remark 3.4]{MTT}, when $A$ is separable, a Herz-Schur multiplier is automatically a Herz-Schur $(A,G,\alpha)$-multiplier. Their argument (for continuous $F$ and $f\in C_c(G,A)$) extends verbatim to arbitrary $(A,G,\alpha)$. As mentioned on \cite[Page 403]{MTT}, when $A=\bC$, both conditions are equivalent to $F$ defining a completely bounded multiplier of the Fourier algebra $A(G)$, as in that case, the associated maps on $C^*_\lm(G)$ admit canonical weak* continuous extensions to $VN(G)$. Such an extension is not ensured to exist in general, hence the two definitions. 

We now show that any element $\xi\in C_c(G,A)$ defines a completely positive Herz-Schur multiplier via $h_\xi(s)(a)=\la\xi,(1\ten a)(\lm_s\ten\alpha_s)\xi\ra$. For discrete dynamical systems, this latter fact follows from \cite[Theorem 2.8]{MSTT} and/or \cite[Theorem 4.8]{BC}.   

\begin{prop}\label{p:CP} Let $(A,G,\alpha)$ be a $C^*$-dynamical system. For each $\xi\in C_c(G,A)$, the function $h:G\rightarrow\mc{CB}(A)$ given by
$$h(s)(a)=\la\xi,(1\ten a)(\lm_s\ten\alpha_s)\xi\ra, \ \ \ s\in G, \ a\in A,$$ 
defines a normal completely positive map $\Theta(h)$ on $(G\ltimes A)''$ satisfying $\norm{\Theta(h)}_{cb}=\norm{h(e)}$, 
\begin{equation}\label{e:map2}\Theta(h)(\alpha(a)(\lm_s\ten 1))=\alpha(h(s)(a))(\lm_s\ten 1),  \ \ \ a\in A, \ s\in G,\end{equation}
and
\begin{equation}\label{e:map}\Theta(h)(\alpha \times\lambda(f))=\alpha \times\lambda(h\cdot f),  \ \ \ f\in C_c(G,A).\end{equation}
When $(A,G,\alpha)=(M_c,G,\alpha)$ for a $W^*$-dynamical system $(M,G,\alpha)$, then 
\begin{equation}\label{e:map3}\Theta(h)(\alpha(x)(\lm_s\ten 1))=\alpha(h(s)(x))(\lm_s\ten 1),  \ \ \ x\in M, \ s\in G.\end{equation}
\end{prop}

\begin{proof} We first consider the map at the level of $B(G\ltimes_f A)$. Let $\Phi\in B(G\ltimes_f A)^+$, and let $\varphi, \sigma$ be as in Equation (\ref{eqn: FS relation}) in Subsection \ref{subsection: ds}. We claim that $h^*\cdot\Phi\in B(G\ltimes_f A)^+$, where 
$$h^*\cdot\Phi\ni G\ni s\mapsto h(s)^*(\Phi(s))\in A^*.$$
By \cite[Proposition 7.6.8]{Ped}, it suffices to show
$$\sum_{j,k=1}^n\la (h^*\cdot\Phi)(t_j^{-1}t_k),\alpha_{t_j^{-1}}(a_j^*a_k)\ra\geq0$$
for any $t_1,...,t_n\in G$ and $a_1,...,a_n\in A$. We compute,
\begin{align*}&\sum_{j,k=1}^n\la (h^*\cdot\Phi)(t_j^{-1}t_k),\alpha_{t_j^{-1}}(a_j^*a_k)\ra
=\sum_{j,k=1}^n\la \Phi(t_j^{-1}t_k),h(t_j^{-1}t_k)(\alpha_{t_j^{-1}}(a_j^*a_k))\ra\\
&=\sum_{j,k=1}^n\la \Phi(t_j^{-1}t_k),\la\xi,(1\ten\alpha_{t_j^{-1}}(a_j^*a_k))(\lm_{t_j^{-1}t_k}\ten\alpha_{t_j^{-1}t_k})\xi\ra\ra\\
&=\sum_{j,k=1}^n\la \Phi(t_j^{-1}t_k),\la(1\ten\alpha_{t_j^{-1}}(a_j))(\lm_{t_j}\ten 1)\xi,(1\ten\alpha_{t_j^{-1}}(a_k))(\lm_{t_k}\ten\alpha_{t_j^{-1}t_k})\xi\ra\ra\\
&=\sum_{j,k=1}^n\la \Phi(t_j^{-1}t_k),\alpha_{t_j^{-1}}(\la(1\ten a_j)(\lm_{t_j}\ten \alpha_{t_j})\xi,(1\ten a_k)\lm_{t_k}\ten\alpha_{t_k}\xi\ra)\ra\\
&=\int_G\sum_{j,k=1}^n\la \Phi(t_j^{-1}t_k),\alpha_{t_j^{-1}}(\alpha_{t_j}(\xi(t_j^{-1}s))^*a_j^*a_k\alpha_{t_k}(\xi(t_k^{-1}s)))\ra \, ds\\
&=\int_G\sum_{j,k=1}^n\la \vphi,\pi(\alpha_{t_j^{-1}}(\alpha_{t_j}(\xi(t_j^{-1}s))^*a_j^*a_k\alpha_{t_k}(\xi(t_k^{-1}s))))\sigma(t_j^{-1}t_k)\ra \, ds\\
&=\int_G\sum_{j,k=1}^n\la \vphi,\sigma(t_j^{-1})\pi(\alpha_{t_j}(\xi(t_j^{-1}s))^*a_j^*a_k\alpha_{t_k}(\xi(t_k^{-1}s)))\sigma(t_k)\ra \, ds\\
&=\int_G\vphi\bigg(\bigg(\sum_{j=1}^na_j\alpha_{t_j}(\xi(t_j^{-1}s))\sigma(t_j)\bigg)^*\bigg(\sum_{k=1}^na_k\alpha_{t_k}(\xi(t_k^{-1}s))\sigma(t_k)\bigg)\bigg)\ra \, ds\\
&\geq0.
\end{align*}
We therefore obtain a well-defined linear map on $B(G\ltimes_f A)=\mathrm{span}B(G\ltimes_f A)^+$ by the Jordan decomposition. 

Since $(M_n(\bC)\ten A,G,\id_{M_n}\ten\alpha)$ is a $C^*$-dynamical system satisfying $M_n(\bC)\ten(G\ltimes_f A)\cong G\ltimes_f(M_n(\bC)\ten A)$ canonically (by \cite[Lemma 2.75]{W}), and since \cite[Proposition 7.6.8]{Ped} applies to any $C^*$-dynamical system, the matricial analogue of the above argument together with the previous identification shows that the linear map 
$$h^*:B(G\ltimes_f A)\ni \Phi\mapsto h^*\cdot \Phi\in B(G\ltimes_f A)$$
is completely positive. Moreover, since $h$ is compactly supported and compactly supported elements of $B(G\ltimes_f A)^+$ lie in $A(G\ltimes_f A)^+$ \cite[Lemma 7.7.6]{Ped}, it follows that
$$h^*:B(G\ltimes_f A)\ni \Phi\mapsto h^*\cdot \Phi\in A(G\ltimes_f A)$$
Since $A(G\ltimes_f A)\subseteq B(G\ltimes_f A)$ and $A(G\ltimes_f A)=(G\ltimes A)''_*\subseteq (G\ltimes A)^*$, by restriction, $h^*$ induces a  completely positive map on $A(G\ltimes_f A)$, whose adjoint $\Theta(h)$ is normal and completely positive on $(G\ltimes A)''$ . Moreover, for each $a\in A$, $s\in G$ and $v\in A(G\ltimes_f A)$,
\begin{align*}\la\Theta(h)(\alpha(a)(\lm_s\ten 1)),v\ra&=\la \alpha(a)(\lm_s\ten 1),h^*(v)\ra\\
&=\la a,h(s)^*(v(s))\ra\\
&=\la h(s)(a),v(s)\ra\\
&=\la\alpha(h(s)(a))(\lm_s\ten 1),v\ra.
\end{align*}
Hence, $\Theta(h)$ satisfies equation (\ref{e:map2}). A similar argument shows that
$$\Theta(h)(\alpha\times\lambda(f))=\alpha\times\lambda(h\cdot f),  \ \ \ f\in C_c(G,A),$$
where $h\cdot f(s)=h(s)(f(s))$, $s\in G$. Taking a bai $(a_i)$ for $A$ which converges strictly (and hence weak*) to the identity of the non-degenerate representation space $H$ of $A$, we have
$$\Theta(h)(1_{(G\ltimes A)''})=w^*\lim_i\Theta(h)(\alpha(a_i))=w^*\lim_i\alpha(h(e)(a_i))=\alpha(\la\xi,\xi \ra).$$ 
By complete positivity, 
$$\norm{\Theta(h)}_{cb}=\norm{\Theta(h)(1)}=\norm{\la\xi,\xi\ra}=\norm{h(e)}.$$

When $(A,G,\alpha)=(M_c,G,\alpha)$ for a $W^*$-dynamical system $(M,G,\alpha)$, then equation (\ref{e:map3}) follows from (\ref{e:map2}), weak* density of $M_c$ in $M$ and normality of $\Theta(h)$ and $\alpha$. Note that in this case we view
$$h(s)(x)=\la\xi,(1\ten x)(\lm_s\ten\alpha_s)\xi\ra\in M$$
in the obvious way as $\xi\in L^2(G,M_c)\subseteq L^2(G,M)$.
\end{proof}

\begin{remark}
For $\xi \in C_c(G, \ell^2(A))$, the function $h(s)(a) = \la \xi, (1 \ten 1 \ten a)(\lm_s \ten 1 \ten \alpha_s) \xi \ra$ also satisfies the conclusions of Proposition \ref{p:CP}. This may be seen by applying Proposition \ref{p:CP} to the functions $h_k$ associated to $\xi_k = P_k \circ \xi$, where $P_k: \ell^2(A) \to A$ is the canonical $k^{th}$ coordinate projection. Then $h(s) = \sum_{k=1}^\infty h_k(s)$ and $\Theta(h) = \sum_{k=1}^\infty \Theta(h_k)$.
\end{remark}

If the range of $\xi$ in Lemma \ref{p:CP} lies in $Z(A)$, then $\Theta(h)$ admits an explicit representation in terms of the fundamental unitary $W_\alpha$, which we now show. It is not clear whether this particular representation is valid for all $\xi\in C_c(G,A)$, although related representations are known to exist at the level of equivariant representations of discrete dynamical systems (see the proof of \cite[Theorem 4.8]{BC}).

In the following, $\om_\xi \ten_\alpha \id$ denotes the map $B(L^2(G,A) \ten_\alpha L^2(G,H)) \to B(L^2(G,H))$, defined so that for $T \in B(L^2(G,A) \ten_\alpha L^2(G,H))$, $(\om_\xi \ten_\alpha \id)(T)$ is the operator in $B(L^2(G,H))$ determined by the sesquilinear form $(\eta_1,\eta_2) \mapsto \la \xi \ten_\alpha \eta_1, T(\xi \ten_\alpha \eta_2) \ra$.

\begin{prop}\label{p:CP2} Let $(A,G,\alpha)$ be a $C^*$-dynamical system. Let $\xi\in C_c(G,Z(A))$ and
$$h(s)=\la\xi,(\lm_s\ten\alpha_s)\xi\ra, \ \ \ s\in G.$$ 
be the associated positive type function. Viewing $h: G \to \mc{CB}(A)$ via multiplication, $h(s)(a) = h(s)a$, the Herz-Schur multiplier $\Theta(h)$ satisfies
\begin{equation}\label{e:rep}\Theta(h)(x)=(\om_\xi\ten_\alpha\id)(W_\alpha^*(1\ten x)W_\alpha), \ \ \ x\in (G\ltimes A)''.\end{equation}
\end{prop}

\begin{proof} By equation (\ref{e:map})
$$\Theta(h)(\alpha\times\lm(f))=\int_G\alpha(h(s)f(s))(\lm_s\ten1) \ ds, \ \ \ f\in C_c(G,A).$$
Represent $A\subseteq\BH$ non-degenerately and view $G\ltimes A\subseteq\mc{B}(\LT\ten H)$. Fix $\eta\in C_c(G,H)$. Then for any $f\in C_c(G,A)$, the commutation relations (\ref{e:fund}) and (\ref{e:center}) imply that
\begin{align*}\la \eta,\Theta(h)((\alpha\times\lm)(f))\eta\ra
&=\int_G \la\eta,\alpha(h(s)f(s))(\lm_s\ten1)\eta\ra \ ds\\
&=\int_G \la\eta,\alpha(\la\xi,\wtal_s\xi\ra)\alpha(f(s))(\lm_s\ten 1)\eta\ra \ ds\\
&=\int_G \la\xi\ten_\alpha\eta,(\wtal_s\xi)\ten_\alpha(\alpha(f(s))(\lm_s\ten 1)\eta)\ra \ ds\\
&=\int_G \la\xi\ten_\alpha\eta,(\wtal_s\xi)\ten_\alpha(\lm_s\ten1)(\alpha(\alpha_{s^{-1}}(f(s)))\eta)\ra \ ds\\
&=\int_G \la\xi\ten_\alpha\eta,(\wtal_s\ten_\alpha(\lm_s\ten1))(\xi\ten_\alpha(\alpha(\alpha_{s^{-1}}(f(s)))\eta)\ra \ ds\\
&=\int_G \la\xi\ten_\alpha\eta,W_\alpha^*(1\ten(\lm_s\ten 1))W_\alpha(\xi\ten_\alpha(\alpha(\alpha_{s^{-1}}(f(s)))\eta)\ra \ ds\\
&=\int_G \la W_\alpha(\xi\ten_\alpha\eta),(1\ten(\lm_s\ten 1)\alpha(\alpha_{s^{-1}}(f(s))))W_\alpha(\xi\ten_\alpha\eta)\ra \ ds\\
&=\int_G \la W_\alpha(\xi\ten_\alpha\eta),(1\ten\alpha(f(s))(\lm_s\ten1))W_\alpha(\xi\ten_\alpha\eta)\ra \ ds\\
&=\la W_\alpha(\xi\ten_\alpha\eta),(1\ten(\alpha\times\lm)(f))W_\alpha(\xi\ten_\alpha\eta)\ra\\
&=\la \xi\ten_\alpha \eta,W_\alpha^*(1\ten(\alpha\times\lm)(f))W_\alpha(\xi\ten_\alpha\eta)\ra\\
&=\la\eta,(\om_{\xi}\ten_\al\id)(W_\alpha^*(1\ten(\alpha\times\lm)(f))W_\alpha)\eta\ra
\end{align*}
It follows that
$$\Theta(h)(x)=(\om_\xi\ten_\alpha\id)(W_\alpha^*(1\ten x)W_\alpha), \ \ \ x\in G\ltimes A.$$
By normality, the above representation extends to all $x\in (G\ltimes A)''$.
\end{proof}

Using the ``fundamental unitary'' $W_\alpha$ associated to the $C^*$-dynamical system $(M_c,G,\alpha)$ we now rephrase the convergence in Theorem \ref{t:Reiter} (2) at a Hilbert space level. This characterization is a dynamical systems analogue of the fundamental unitary characterization of (co-)amenability of locally compact (quantum) groups, and it leads to an approximation of the identity of $G\bar{\ltimes} M$ by completely positive Herz-Schur multipliers.

\begin{thm}\label{t:W*amen} Let $(M,G,\alpha)$ be a $W^*$-dynamical system. The following conditions are equivalent:
\begin{enumerate}
\item $(M,G,\alpha)$ is amenable; 
\item there exists a net $(\xi_i)$ in $C_c(G,Z(M)_c)$ such that
\begin{enumerate}
\item $\la\xi_i,\xi_i\ra=1$ for all $i$;
\item $\norm{W_\alpha(\xi_i\ten_\alpha\eta)-\xi_i\cdot\eta}_{L^2(G\times G,H)}\rightarrow0$, $\eta\in L^2(G,H)$;
\end{enumerate}
\item there exists a net $(\xi_i)$ in $C_c(G,M_c)$ such that
\begin{enumerate}
\item $\la\xi_i,\xi_i\ra=1$ for all $i$;
\item $\Theta(h_{\xi_i})\rightarrow\id_{G\bar{\ltimes} M}$ point weak*.
\end{enumerate}
\end{enumerate}
\end{thm}

\begin{proof} $(1)\Rightarrow(2)$: If $(M,G,\alpha)$ is amenable, by Theorem \ref{t:Reiter} there exists a net $(\xi_i)$ in $C_c(G,Z(M)_c)$ such that $\la\xi_i,\xi_i\ra=1$ for all $i$ and $\la\xi_i,\wtal_t\xi_i\ra\rightarrow 1$ weak*, uniformly on compact subsets. Fix a non-degenerate normal representation $M \subseteq B(H)$. Let $\eta=\eta_1\ten\eta_2$ with $\eta_1\in C_c(G)$, and $\eta_2\in H$. Let $\om_{\eta_2}$ be the associated vector functional on $B(H)$. By norm continuity of the action $G\acts M_*$ and Lemma \ref{l:bdd w*=ucc}, it follows that
$$\om_{\eta_2}(\la\wtal_{t^{-1}}\xi_i-\xi_i,\wtal_{t^{-1}}\xi_i-\xi_i\ra)=\om_{\eta_2}(\alpha_{t^{-1}}(\la\xi_i,\xi_i\ra)-2\mathrm{Re}\la\xi_i,\wtal_{t^{-1}}\xi_i\ra+\la\xi_i,\xi_i\ra)\rightarrow0$$
uniformly on compact subsets of $G$. Hence,
\begin{align*}&\norm{W_\alpha(\xi_i\ten_\alpha\eta)-\xi_i\cdot\eta}_{L^2(G\times G,H)}^2=\iint\norm{(\wtal_{t^{-1}}\xi_i(s)-\xi_i(s))\eta(t)}^2_H \ ds \ dt\\
&=\iint|\eta_1(t)|^2\norm{(\wtal_{t^{-1}}\xi_i(s)-\xi_i(s))\eta_2}^2_H \ ds \ dt\\
&=\iint|\eta_1(t)|^2\la\eta_2,(\wtal_{t^{-1}}\xi_i(s)-\xi_i(s))^*(\wtal_{t^{-1}}\xi_i(s)-\xi_i(s))\eta_2\ra_H \ ds \ dt\\
&=\int|\eta_1(t)|^2\la\eta_2,\la\wtal_{t^{-1}}\xi_i-\xi_i,\wtal_{t^{-1}}\xi_i-\xi_i\ra\eta_2\ra_H \ dt\\
&\rightarrow 0.
\end{align*}
Since linear combinations of simple tensors $\eta_1\ten\eta_2$ with $\eta_1\in C_c(G)$ and $\eta_2\in H$ are dense in $L^2(G,H)$, boundedness of $W_\alpha$ and $(\xi_i)$, together with the inequality $\norm{\xi\cdot\eta}\leq\norm{\xi}\norm{\eta}$ (which follows from (\ref{e:pi})) show that 
$$\norm{W_\alpha(\xi\ten_\alpha\eta)-\xi\cdot\eta}_{L^2(G\times G,H)}\rightarrow0,$$
for all $\eta\in L^2(G,H)$.

$(2)\Rightarrow(3)$: Pick a net $(\xi_i)$ in $C_c(G,Z(M)_c)$ satisfying (2). If $(h_i)$ denotes the corresponding positive type functions in $C_c(G,Z(M)_c)$, then Propositions \ref{p:CP} and \ref{p:CP2} applied to the $C^*$-dynamical system $(M_c,G,\alpha)$ imply that
$$\Theta(h_i)(x)=(\om_{\xi_i}\ten_\alpha\id)(W_\alpha^*(1\ten x)W_\alpha), \ \ \ x\in G\bar{\ltimes} M,$$
and $\norm{\Theta(h_i)}_{cb}=\norm{\la\xi_i,\xi_i\ra}=1$.

By boundedness of $(\Theta(h_i))$, it suffices to show that for any $x \in G \bar{\ltimes} M$ and $\eta \in L^2(G,H)$, \[| \la \eta, \Theta(h_i)(x) \eta \ra - \la \eta, x\eta \ra | \to 0.\]
To show this, note that by the representation (\ref{e:rep}),
\[ \la \eta, \Theta(h_i)(x) \eta \ra = \la W_\alpha(\xi_i \ten_\alpha \eta), (1 \ten x)W_\alpha(\xi_i \ten_\alpha \eta) \ra.\]
On the other hand, since $\la \xi,\xi \ra = 1$,
\[ \la \eta,x\eta \ra = \la \xi_i \cdot \eta, (1 \ten x)(\xi_i \cdot \eta) \ra.\]
Thus, the result follows from condition (2)(b) and the general fact that when $(y_i),(z_i)$ are bounded nets in a Hilbert space such that $y_i-z_i \to 0$, then for any operator $T$, $|\la y_i,T(y_i) \ra - \la z_i,T(z_i) \ra | \to 0$.

$(3)\Rightarrow(1)$:  By property (3), there exists a net $(\xi_i)$ of compactly supported positive type functions in $C_c(G,M_c)$ with $\la\xi_i,\xi_i\ra=1$ and whose corresponding Herz--Schur multipliers $\Theta(h_{\xi_i})$ converge to $\id_{G\bar{\ltimes}M}$ point weak*. By Proposition \ref{p:CP} applied to $(M_c,G,\alpha)$, it follows that 
$$\alpha(h_{\xi_i}(s)(x))(\lm_s\ten 1)=\Theta(h_{\xi_i})(\alpha(x)(\lm_s\ten 1))\xrightarrow{w^*}\alpha(x)(\lm_s\ten 1),$$
and therefore $\alpha(h_{\xi_i}(s)(x))\rightarrow\alpha(x)$ weak*, for each $x\in M$ and $s\in G$. Since $\alpha:M\rightarrow\LI\oten M$ is a weak*-weak* homeomorphism onto its range, it follows that 
$$\la\xi_i,(1\ten x)(\lm_s\ten\alpha_s)\xi_i\ra=h_{\xi_i}(s)(x)\xrightarrow{w^*} x, \ \ \ x\in M, \ s\in G.$$
Hence, $(\xi_i)$ satisfies condition (7) of \cite[Proposition 3.12]{BEW}. Since the implication $(7)\Rightarrow(8)$ of \cite[Proposition 3.12]{BEW} is valid for arbitrary locally compact groups, it follows that $(M,G,\alpha)$ is amenable.
\end{proof}

\section{Amenable $C^*$-dynamical systems}

In their recent study of amenability and weak containment for $C^*$-dynamical systems \cite{BEW}, Buss, Echterhoff and Willett introduced the following definitions.

\begin{defn}\cite{BEW} Let $(A,G,\alpha)$ be a $C^*$-dynamical system. Then $(A,G,\alpha)$ is:
\begin{itemize}
\item \textit{von Neumann amenable} if the universal $W^*$-dynamical system $(A_\alpha'',G,\alpha)$ is amenable;
\item \textit{amenable} if there exists a net of norm-continuous, compactly supported, positive type functions $h_i:G\rightarrow Z(A_\alpha'')$ such that $\norm{h_i(e)}\leq 1$ for all $i$, and $h_i(s)\rightarrow 1$ weak* in $A_\alpha''$, uniformly for $s$ in compact subsets of $G$;
\item \textit{strongly amenable} if there exists a net $(h_i)\in P_1(A,G,\alpha)\cap C_c(G,Z(M(A)))$ such that $h_i(s)\rightarrow 1$ strictly, uniformly on compact subsets of $G$.
\end{itemize}
\end{defn}

It was shown in \cite[Proposition 3.12]{BEW} that amenability always implies von Neumann amenability and that the conditions are equivalent when $G$ is exact. It follows from Theorem \ref{t:Reiter} (and the subsequent Remark \ref{r:Reiter}) that amenability and von Neumann amenability coincide for arbitrary $C^*$-dynamical systems.

Strong amenability always implies amenability \cite[Remark 3.6]{BEW}, however, results of Suzuki \cite{Suz} imply that for non-commutative $A$, amenability is, in general, strictly weaker than strong amenability. For commutative $A$ and discrete $G$, strong amenability coincides with amenability by \cite[Th\'{e}or\`{e}me 4.9]{AD87}. We show in Corollary \ref{c:SA} that the two notions coincide for arbitrary commutative $C^*$-dynamical systems, thus answering \cite[Question 8.1]{BEW} in the affirmative. 

Another approach to amenability is through Exel's approximation property of Fell bundles over discrete groups \cite{Exel}. This property was later generalized by Exel and Ng in \cite{EN} to Fell bundles over locally compact groups. Specializing to the case of crossed products of $C^*$-dynamical systems, they defined the \textit{C-approximation property} of $(A,G,\alpha)$ to be the existence of nets $(\xi_i)$ and $(\eta_i)$ in $C_c(G,A)$ for which $\norm{\la\xi_i,\xi_i\ra}\norm{\la\eta_i,\eta_i\ra}\leq C$ and for any $f\in C_c(G,A)$
$$\int_G\xi_i(t)^*f(s)\alpha_s(\eta_i(s^{-1}t)) \ dt\rightarrow f(s)$$
in norm, uniformly in $(s,f(s))$. If one can take $\eta_i=\xi_i$, then $(A,G,\alpha)$ has the \textit{C-positive approximation property}. Exel and Ng showed that when $A$ is nuclear and $G$ is discrete, then the approximation property implies amenability of $(A,G,\alpha)$, and conversely, the two notions are equivalent whenever $G$ is discrete and $A$ is commutative or finite-dimensional (see \cite[Section 4]{EN}).

In \cite{BC3}, B\'{e}dos and Conti generalized this notion by defining the \textit{C-weak approximation property} as the existence of an equivariant representation $(\rho,v)$ of $(A,G,\alpha)$ on a Hilbert $A$-module $E$ (see, e.g., \cite[Page 40]{BC}), and nets $(\xi_i)$ and $(\eta_i)$ in $C_c(G,E)$ for which $\norm{\la\xi_i,\xi_i\ra}\norm{\la\eta_i,\eta_i\ra}\leq C$ and
$$\la\xi_i,\rho(a)v(s)\eta_i\ra \rightarrow a, \ \ \ a \in A,$$
uniformly for $s$ in compact subsets of $G$. (This property was defined for discrete dynamical systems in \cite{BC3}, the definition above being the natural generalization.) Again, if one can take $\xi_i=\eta_i$, then $(A,G,\alpha)$ has the \textit{C-positive weak approximation property}. For discrete dynamical systems with $A$ unital, B\'{e}dos and Conti showed that the weak approximation property implies that the full and reduced crossed products coincide \cite[Theorem 4.32]{BC2}. 

By \cite[Theorem 3.28]{BEW}, it follows that the $C$-positive approximation property implies amenability. Below we establish a partial converse, showing the equivalence of amenability and a particular case of the 1-positive weak approximation property of B\'{e}dos and Conti, when $E=\ell^2(A)$. When $A$ is commutative, or more generally, when $Z(A^{**})=Z(A)^{**}$, we can take $E=A$, in which case amenability is equivalent to the 1-positive approximation property. This is a consequence of our main result of this section:

\begin{thm}\label{t:C*amen} Let $(A,G,\alpha)$ be a $C^*$-dynamical system. The following conditions are equivalent:
\begin{enumerate}
\item $(A,G,\alpha)$ is amenable;
\item there exists a net $(\xi_i)$ in $C_c(G,\ell^2(A))$ such that
\begin{enumerate}
\item $\la\xi_i,\xi_i\ra\leq 1$ for all $i$;
\item $h_{\xi_i}(e)\rightarrow\id_{A}$ in the point norm topology, and
\item $\Theta(h_{\xi_i})\rightarrow\id_{G\ltimes A}$ in the point norm topology,
\end{enumerate}
where $h_{\xi_i}(s)(a)=\la\xi_i,(1\ten 1\ten a)(\lm_s\ten1\ten \alpha_s)\xi_i\ra$ are the associated completely positive Herz-Schur multipliers;
\item there exists a net $(\xi_i)$ in $C_c(G,\ell^2(A))$ such that $\la\xi_i,\xi_i\ra\leq 1$ for all $i$ and 
$$\norm{h_{\xi_i}(s)(f(s))-f(s)}\rightarrow0, \ \ \ f\in C_c(G,A),$$
uniformly for $s$ in compact subsets of $G$;
\item $(A,G,\alpha)$ is von Neumann amenable.
\end{enumerate}
Moreover, when $Z(A^{**})=Z(A)^{**}$, the net $(\xi_i)$ can be chosen in $C_c(G,Z(A))_{\|\cdot\|_{L^2(G,Z(A))} \leq 1}$, in which case $h_i(s)(a)=a\la\xi_i,(\lm_s\ten\alpha_s)\xi_i\ra$, $s\in G$, $a\in A$.
\end{thm}

The outline of the proof is as follows: we first use the Kaplansky density theorem for Hilbert modules to obtain a $C^*$-Reiter type property from amenability, which is then used to deduce the Herz-Schur multiplier convergence (Proposition \ref{p:C*amen}). The equivalence of (2) and (3) follows from a more general equivalence at the level of compactly supported completely positive multipliers (Theorem \ref{t:posdef}). The final step uses the techniques from \cite[Lemma 6.5]{ABF} to deduce von Neumann amenability from the weak approximation property in (3), at which point amenability follows from Theorem \ref{t:Reiter}.

We begin with the following estimate, which will be used several times in the sequel.

\begin{lem}\label{l:Hibmod} Let $A$ be a $C^*$-algebra and $E$ be an inner product $A$-module. Then for any state $\mu\in A^*$ and $\xi,\xi',\eta,\eta'\in E$, 
$$|\mu(\la\xi,\xi'\ra-\la\eta,\eta'\ra)|\leq\norm{\la\xi,\xi\ra}^{1/2}\mu(\la\xi'-\eta',\xi'-\eta'\ra)^{1/2}+\norm{\la\eta',\eta'\ra}^{1/2}\mu(\la\xi-\eta,\xi-\eta\ra)^{1/2}.$$
\end{lem}

\begin{proof} By the Schwarz inequality for completely positive maps, for any $a\in A$ we have $|\mu(a)|^2\leq\mu(a^*a), \mu(aa^*)$. Combining this with the Cauchy-Schwarz inequality \cite[Proposition 1.1]{Lance} for $E$, we have
\begin{align*}|\mu(\la\xi,\xi'\ra-\la\eta,\eta'\ra)|&=|\mu(\la\xi,\xi'-\eta'\ra+\la\xi-\eta,\eta'\ra)|\\
&\leq\mu(\la\xi'-\eta',\xi\ra\la\xi,\xi'-\eta'\ra)^{1/2}+\mu(\la\xi-\eta,\eta'\ra\la\eta',\xi-\eta\ra)^{1/2}\\
&\leq\norm{\la\xi,\xi\ra}^{1/2}\mu(\la\xi'-\eta',\xi'-\eta'\ra)^{1/2}+\norm{\la\eta',\eta'\ra}^{1/2}\mu(\la\xi-\eta,\xi-\eta\ra)^{1/2}.
\end{align*}
\end{proof}

The next lemma is surely known. We include a proof for completeness.

\begin{lem}\label{l:hten} Let $G$ be a locally compact group and $A$ be a $C^*$-algebra. Then $L^2(G,A)\cong L^2(G)_c\hten A$ completely isometrically.
\end{lem}

\begin{proof} Let $(e_i)_{i\in I}$ be an orthonormal basis of $L^2(G)$. Given $\xi_1,...,\xi_n\in L^2(G)$ and $a_1,...,a_n\in A$, for each $i$, let $b_i=\sum_{k=1}^n \la e_i,\xi_k\ra a_k$. Then, on the one hand
$$\norm{\sum_{k=1}^n \xi_k\ten a_k}_h=\norm{\sum_{i\in I}\sum_{k=1}^n\la e_i,\xi_k\ra e_i \ten a_k}_h=\norm{\sum_{i\in I}e_i\ten b_i}_h=\norm{\sum_{i\in I}b_i^*b_i}^{1/2}.$$
On the other hand
\begin{align*}\norm{\sum_{k=1}^n \xi_k\ten a_k}_{L^2(G,A)}&=\norm{\sum_{k,l=1}^n \la\xi_k,\xi_l\ra a_k^*a_l}^{1/2}=\norm{\sum_{i\in I}\sum_{k,l=1}^n \la\xi_k,e_i\ra\la e_i,\xi_l\ra a_k^*a_l}^{1/2}\\
&=\norm{\sum_{i\in I}b_i^*b_i}^{1/2}.
\end{align*}
Thus there is an isometric isomorphism $\theta: L^2(G)_c \ten^h A \to L^2(G,A)$ acting as the identity on simple tensors. Equipping the space $L^2(G)_c \ten^h A$ with the canonical $C^*$-$A$-module structure (see \cite[Theorem 8.2.11]{BLM}), standard calculations show that $\theta$ is an $A$-module map satisfying $\theta(x \la y,z \ra) = \theta(x)\la \theta(y), \theta(z) \ra$ for all $x,y,z \in L^2(G)_c \otimes^h A$. Thus, if we equip $L^2(G,A)$ with its canonical operator space structure (see \cite[Section 8.2]{BLM}), it follows by \cite[Lemma 8.3.2]{BLM} that $\theta$ is completely isometric.
\end{proof}

Let $A$ be a $C^*$-algebra. The self-dual completion of a Hilbert $A$-module $E$ is the space $E'':=B_A(E,A^{**})$ of bounded $A$-module maps from $E$ into $A^{**}$. By \cite[Corollary 4.3]{P} (see also \cite[Proposition 2.2]{Zettl}) there is a Hilbert $A^{**}$-module structure on $E''$, whose norm coincides with the operator norm induced from $B_A(E,A^{**})$.

\begin{lem}\label{l:sdcompletion} Let $G$ be a locally compact group and $A$ be a $C^*$-algebra. The map 
$$j:L^2(G,A^{**})\ni\xi\mapsto\bigg(\eta\mapsto \la\xi,\eta\ra_{A^{**}}=\int_G \xi(s)^*\eta(s) \ ds\bigg)\in L^2(G,A)''$$
is an isometric $A^{**}$-module map.
\end{lem}

\begin{proof} Fix $\xi\in C_c(G,A^{**})$. Then
\begin{align*}\norm{j(\xi)}_{L^2(G,A)''}&=\sup\{\norm{\la\xi,\eta\ra_{A^{**}}}\mid \eta\in L^2(G,A), \ \norm{\eta}\leq1\}\\
&\leq\sup\{\norm{\la\xi,\eta\ra_{A^{**}}}\mid \eta\in L^2(G,A^{**}), \ \norm{\eta}\leq 1\}\\
&=\norm{\xi}_{L^2(G,A^{**})}.
\end{align*}
For the reverse inequality, first note that by self-duality of the Haagerup tensor product \cite[Corollary 3.4]{BS}, the canonical inclusion 
$$L^2(G)_c\hten A^{**}=L^2(G)_c^{**}\hten A^{**}\hookrightarrow L^2(G)_c^{**}\whten A^{**}$$
is a complete isometry. Further, by \cite[Theorem 5.7]{ER2}, the canonical injection
$$L^2(G)_c^{**}\whten A^{**}\hookrightarrow (L^2(G)_c^*\whten A^*)^*=(L^2(G)_c\hten A)^{**}$$
is a complete isometry. Hence, $L^2(G)_c\hten A^{**}\subseteq(L^2(G)_c\hten A)^{**}$, canonically. Let $(\xi_i)$ be a net in $(L^2(G)_c\hten A)_{\norm{\cdot}\leq\norm{\xi}}$ which converges to $\xi$ in the weak* topology of  $(L^2(G)_c^*\whten A^*)^*$. Then for every $\chi\in L^2(G)$ and $\mu\in A^*$ we have
$$\la\xi,\chi\ten\mu\ra=\lim_i\int_G \la \xi_i(s),\mu\ra \overline{\chi(s)} \ ds,$$
uniformly for $\mu$ in compact subsets of $A^*$ (by Lemma \ref{l:bdd w*=ucc}). Let $\chi=\chi_{\mathrm{supp}(\xi)}\in L^2(G)$. Then for every $\mu\in A^*$, the set $\{\mu\cdot\xi(s)^*\mid s\in G\}$ is norm compact in $A^*$, so that
\begin{align*}\mu(\la\xi,\xi_i\ra_{A^{**}})&=\int_G \la\xi(s)^*\xi_i(s),\mu\ra\overline{\chi(s)} \ ds\\
&=\int_G \la\xi_i(s),\mu\cdot\xi(s)^*\ra\overline{\chi(s)} \ ds\\
&\rightarrow\int_G \la\xi(s),\mu\cdot\xi(s)^*\ra\overline{\chi(s)} \ ds\\
&=\mu(\la\xi,\xi\ra_{A^{**}}).
\end{align*}
Hence, $\la\xi,\xi_i\ra_{A^{**}}\rightarrow\la\xi,\xi\ra_{A^{**}}$ weak* in $A^{**}$, and so
\begin{align*}\norm{\la\xi,\xi\ra_{A^{**}}}&\leq\limsup_i\norm{\la\xi,\xi_i\ra_{A^{**}}}\leq\limsup_i\norm{j(\xi)}_{L^2(G,A)''}\norm{\xi_i}_{L^2(G,A)}\\
&\leq\norm{j(\xi)}_{L^2(G,A)''}\norm{\xi}_{L^2(G,A^{**})},
\end{align*}
which implies that $\norm{\xi}_{L^2(G,A^{**})}\leq\norm{j(\xi)}_{L^2(G,A)''}$.
\end{proof}

\begin{prop}\label{p:C*amen} Let $(A,G,\alpha)$ be an amenable $C^*$-dynamical system. Then there exists a net $(h_i)$ of continuous compactly supported completely positive Herz-Schur multipliers satisfying
\begin{enumerate}
\item $\norm{h_i(e)}_{cb}\leq 1$ for all $i$;
\item $h_i(e)\rightarrow\id_{A}$ in the point norm topology;
\item $\Theta(h_i)\rightarrow\id_{G\ltimes A}$ in the point norm topology.
\item $h_i(s)(a)=\la\xi_i,(1\ten 1\ten a)(\lm_s\ten 1\ten\alpha_s)\xi_i\ra_A$, for a contractive net $(\xi_i)$ in $C_c(G,\ell^2(A))$.
\end{enumerate}
When $Z(A^{**})=Z(A)^{**}$, the net $(\xi_i)$ can be chosen in $C_c(G,Z(A))$, in which case $h_i(s)(a)=a\la\xi_i,(\lm_s\ten\alpha_s)\xi_i\ra$, $s\in G$, $a\in A$.
\end{prop}

\begin{proof} By Theorem \ref{t:Reiter}, amenability of $(A,G,\alpha)$ implies the existence  a net $(\xi_i)$ in $C_c(G,Z(A_\alpha'')_c)$ whose corresponding positive type functions $h_i(s)=\la\xi_i,(\lm_s\ten\alpha_s)\xi_i\ra$ satisfy $h_i(e)=\la\xi_i,\xi_i\ra=1$ for all $i$, $\lim_i h_i(s)=1$ weak*, uniformly on compact subsets. 

Pick $\eta\in C_c(G)_{\norm{\cdot}_2=1}$ and let $\xi_i'=(1\ten z)\xi_i+\eta\ten (1-z)\in C_c(G,Z(A^{**}))$. Then
$$\la\xi_i',\xi_i'\ra=\int_G z\xi_i(s)^*\xi(s)+|\eta(s)|^2(1-z) \ ds=z\la\xi_i,\xi_i\ra+\norm{\eta}^2(1-z)=1.$$ 
By Lemma \ref{l:sdcompletion}, $(j(\xi'_i))$ is a net in the unit ball of $L^2(G,A)''$. By the Kaplanksy density theorem for Hilbert $C^*$-modules \cite[Corollary 2.7]{Zettl}, for each $i$, there exists a net $(\xi_{i,j})$ in $C_c(G,A)_{\norm{\cdot}_{L^2(G,A)}\leq 1}$ such that 
$$\mu(\la j(\xi'_i)-j(\xi_{i,j}),j(\xi'_i)-j(\xi_{i,j})\ra_{L^2(G,A)''} )^{1/2}=\mu(\la \xi'_i-\xi_{i,j},\xi'_i-\xi_{i,j}\ra_{A^{**}} )^{1/2}\rightarrow 0, \ \ \ \mu\in (A^*)^+,$$
where the first equality uses that $j$ is an isometric $A^{**}$-module map (Lemma \ref{l:sdcompletion}). We now observe two consequences of this approximation which will be combined into a single convexity argument to yield the desired properties (2)-(4) (property (1) being automatic).

First, for any state $\mu\in A^*$, applying Lemma \ref{l:Hibmod} to the inner product $A^{**}$-module $E=C_c(G,A^{**})$, we have
\begin{align*}\mu(1-\la\xi_{i,j},\xi_{i,j}\ra)&=\mu(\la\xi'_{i},\xi'_{i}\ra-\la\xi_{i,j},\xi_{i,j}\ra)\\
&\leq\norm{\la\xi_i',\xi_i'\ra}\mu(\la\xi_i'-\xi_{i,j},\xi_i'-\xi_{i,j}\ra)^{1/2}+\norm{\la\xi_{i,j},\xi_{i,j}\ra}\mu(\la\xi_i'-\xi_{i,j},\xi_i'-\xi_{i,j}\ra)^{1/2}\\
&\leq 2\mu(\la\xi_i'-\xi_{i,j},\xi_i'-\xi_{i,j}\ra)^{1/2}\\
&\xrightarrow{j}0.
\end{align*}
Thus, $\la\xi_{i,j},\xi_{i,j}\ra\rightarrow 1$ weak* in $A^{**}$, where we are considering the doubly-indexed net as in \cite[pg.\ 69]{Kelley}. Then for each $i$ and any state $\mu\in A^*$, 
$$\mu(\la (1\ten a)(\xi_{i,j}-\xi_i'),(1\ten a)(\xi_{i,j}-\xi_i')\ra)^{1/2}\leq\norm{a}\mu(\la\xi_{i,j}-\xi_i',\xi_{i,j}-\xi_i'\ra)^{1/2}\xrightarrow{j}0$$
and, similarly,
$$\mu(\la(\xi_{i,j}-\xi_i')(1\ten a),(\xi_{i,j}-\xi_i')(1\ten a)\ra)^{1/2}=(a\cdot\mu\cdot a^*)(\la\xi_{i,j}-\xi_i',\xi_{i,j}-\xi_i'\ra)^{1/2}\xrightarrow{j}0.$$
In addition, as $\xi_i'$ takes values in $Z(A^{**})$, we have $(1\ten a)\xi'_i=\xi'_i(1\ten a)$ for each $i$ and each $a\in A$. Using this, and applying similar estimates from the proof of Lemma \ref{l:Hibmod}, for any state $\mu$ and $a\in A$,
\begin{align*}&\mu(\la\xi_{i,j},(1\ten a^*)\xi_{i,j}\ra-a^*\la\xi_{i,j},\xi_{i,j}\ra)\\
&=\mu(\la(1\ten a)\xi_{i,j},\xi_{i,j}\ra-\la\xi_{i,j}(1\ten a),\xi_{i,j}\ra)\\
&=\mu(\la(1\ten a)(\xi_{i,j}-\xi_i'),\xi_{i,j}\ra+\la(\xi_i'-\xi_{i,j})(1\ten a),\xi_{i,j}\ra)\\
&\leq\norm{\la\xi_{i,j},\xi_{i,j}\ra}^{1/2}\mu(\la (1\ten a)(\xi_{i,j}-\xi_i'),(1\ten a)(\xi_{i,j}-\xi_i')\ra)^{1/2}\\
& \ \ \ \ \ \ \ \ +\norm{\la\xi_{i,j},\xi_{i,j}\ra}^{1/2}\mu(\la(\xi_{i,j}-\xi_i')(1\ten a),(\xi_{i,j}-\xi_i')(1\ten a)\ra)^{1/2}\\
&\rightarrow0,
\end{align*}
Thus, $\la\xi_{i,j},(1\ten a)\xi_{i,j}\ra-a\la\xi_{i,j},\xi_{i,j}\ra\rightarrow0$ weak* in $A^{**}$, and it follows that $\la\xi_{i,j},(1\ten a)\xi_{i,j}\ra\rightarrow a$ weak* in $A^{**}$ for each $a\in A$.

Second, since $(1\ten z)\xi_i'$ is equal to the original $\xi_i\in C_c(G,Z(A_\alpha'')_c)$, for any $\mu\in (A_\alpha'')_*^+=z(A^*)^+$, we have 
$$\mu(\la\xi_i-z\xi_{i,j},\xi_i-z\xi_{i,j}\ra_{A''_\alpha})^{1/2}= \mu(\la \xi_i'-\xi_{i,j}, \xi_i'-\xi_{i,j} \ra_{A^{**}})^{1/2}\rightarrow 0,$$
where $z\xi_{i,j}$ is shorthand for $(1\ten z)\xi_{i,j}$. Fix a state $\mu\in (A_\alpha'')_*^+$, $a\in A$ and let $\eta_i=(1\ten a)^*\xi_i$ and $\eta_{i,j}=(1\ten a)^*z\xi_{i,j}$.  Then, by Lemma \ref{l:Hibmod} applied to the inner product $A''_\alpha$-module $E = C_c(G, A''_\alpha)$,
\begin{align*}&|\mu(\la\xi_i,(1\ten a)(\lm_t\ten\overline{\alpha}_t)\xi_{i}\ra)-\mu(\la z\xi_{i,j},(1\ten a)(\lm_t\ten\overline{\alpha}_t)z\xi_{i,j}\ra)|\\
&=|\mu(\la\eta_i,(\lm_t\ten\overline{\alpha}_t)\xi_{i}\ra)-\mu(\la \eta_{i,j},(\lm_t\ten\overline{\alpha}_t)z\xi_{i,j}\ra)|\\
&\leq\norm{\la\eta_i,\eta_i\ra}^{1/2}\mu\circ\overline{\alpha}_t(\la z\xi_{i,j}-\xi_i,z\xi_{i,j}-\xi_i\ra)^{1/2}\\
& \ \ \ \ \ \ +\norm{\la z\xi_{i,j},z\xi_{i,j}\ra}^{1/2}\mu(\la \eta_{i,j}-\eta_i,\eta_{i,j}-\eta_i\ra)^{1/2}\\
&=\norm{\la(1\ten a)^*\xi_i,(1\ten a)^*\xi_i\ra}^{1/2}\mu\circ\overline{\alpha}_t(\la z\xi_{i,j}-\xi_i,z\xi_{i,j}-\xi_i\ra)^{1/2}\\
& \ \ \ \ \ \ +\mu(\la (1\ten a)^*(z\xi_{i,j}-\xi_i),(1\ten a)^*(z\xi_{i,j}-\xi_i)\ra)^{1/2}\\
&\leq\norm{a}\mu\circ\overline{\alpha}_t(\la z\xi_{i,j}-\xi_i,z\xi_{i,j}-\xi_i\ra)^{1/2}\\
& \ \ \ \ \ \ +\norm{a}\mu(\la z\xi_{i,j}-\xi_i,z\xi_{i,j}-\xi_i\ra)^{1/2}.
\end{align*}
Once again using the norm continuity of the predual action $G\acts (A_\alpha'')_*$ and Lemma \ref{l:bdd w*=ucc}, the above estimates imply that
$$|\mu(\la\xi_i,(1\ten a)(\lm_t\ten\overline{\alpha}_t)\xi_{i}\ra)-\mu(\la z\xi_{i,j},(1\ten a)(\lm_t\ten\overline{\alpha}_t)z\xi_{i,j}\ra)|\xrightarrow{j}0, \ \ \ \mu \in (A_\alpha'')_*,$$
uniformly for $t$ in compact subsets of $G$, and $a$ in bounded subsets of $A$. Putting $h_{i,j}(t)(a)=\la \xi_{i,j},(1\ten a)(\lm_t\ten\alpha_t)\xi_{i,j}\ra$, we obtain a net $(h_{i,j})$ of compactly supported completely positive Herz-Schur multipliers satisfying $\norm{h_{i,j}(e)}_{cb}\leq 1$ and (recalling that each $\xi_i$ takes central values),
\begin{equation*}|\mu(zh_{i,j}(t)(za))-\mu(za)|\leq|\mu(zh_{i,j}(t)(za))-\mu(h_i(t)za)|+|\mu(h_i(t)za)-\mu(za)|\xrightarrow{i,j}0\end{equation*}
for any $\mu\in (A_\alpha'')_*$, uniformly for $(t,a)$ in compact subsets of $G\times A$ (the uniformity on compacta in $A$ coming from the convergence of the second term above).

Fix $f\in C_c(G,zA)$. By \cite[Lemme 3.2]{E} there exists a linear combination $v\in A(G)$ of positive definite functions in $C_c(G)$ such that $v\equiv 1$ on $\mathrm{supp}(f)$. It follows that
$$v\cdot(\alpha\times\lm)(f)=(\alpha\times\lm)(v\cdot f)=(\alpha\times\lm)(f),$$
where $\cdot$ is the canonical action of $A(G)$ on $G\ltimes zA$ via the dual co-action. Given $u\in (G\ltimes zA)^*\subseteq B(G\ltimes_f zA)$, by \cite[Corollary 7.6.9]{Ped}, $v\cdot u$ is a linear combination of compactly supported positive definite functions in $B(G\ltimes_f zA)$. Hence, by \cite[Lemma 7.7.6]{Ped}, 
$$v\cdot u\in A(G\ltimes_f zA)=(G\ltimes zA)''_*\cong (G\bar{\ltimes} A_\alpha'')_*.$$
Then $\{v(s)u(s)\mid s\in G\}$ is a norm compact subset of $(A_\alpha'')_*$, so boundedness of $\norm{h_{i,j}(s)}$, the identification $A(G\ltimes_f zA)=(G\bar{\ltimes} A_\alpha'')_*$ and the weak* convergence $zh_{i,j}(s)(za)\rightarrow za$ imply that 
\begin{align*}\la u,\Theta(zh_{i,j})(\alpha\times\lm(f))\ra&=\la v\cdot u,\Theta(zh_{i,j})(\alpha\times\lm(f))\ra\\
&=\int_G\la v(s)u(s),zh_{i,j}(s)(f(s))\ra \ ds\\
&\rightarrow\int_G\la v(s)u(s),f(s)\ra \ ds\\
&=\la u, (\alpha\times\lm)(f)\ra.
\end{align*}
By boundedness of $(\Theta(zh_{i,j}))$, it follows that $\Theta(zh_{i,j})\rightarrow\id_{G\ltimes zA}$ in the point weak topology. Identifying $A$ with $zA\subseteq A_\alpha''$, as well as the $C^*$-dynamical systems $(A,G,\alpha)\cong (zA,G,\overline{\alpha})$, it follows that $\Theta(h_{i,j})\rightarrow\id_{G\ltimes A}$ in the point weak topology.

Now for every $a_1,...,a_n\in A$, $x_1,...,x_m\in G\ltimes A$, consider the convex set
$$C=\{(h(e)(a_1)-a_1,...,h(e)(a_n)-a_n,\Theta(h)(x_1)-x_1,...,\Theta(h)(x_m)-x_m)\mid h\in\mathrm{conv}\{h_{i,j}\}\},$$
viewed inside the locally convex Hausdorff space 
$$(A,w)\oplus\cdots\oplus(A,w)\oplus (G\ltimes A,w)\oplus\cdots\oplus(G\ltimes A,w),$$
where $w$ denotes the weak topology. By the above analysis, $0$ belongs to the closure of $C$. The standard convexity argument then shows that $0$ belongs to the closure of $C$ where all summands are equipped with the norm topology. It follows that there exists a net $(h_i)$ of continuous compactly supported completely positive Herz-Schur multipliers $h_i:G\rightarrow\mc{CB}(A)$ satisfying properties (1)-(3), and each $h_i\in \mathrm{conv}\{h_{i,j}\}$. To see that (4) holds, use $h_i \in \mathrm{conv}\{h_{i,j}\}$ to write each $h_i$ as
\begin{align*}h_i(s)(a)&=\sum_{k=1}^{n_i}\lm_k\la\xi_{i_k,j_k},(1\ten a)(\lm_s\ten\alpha_s)\xi_{i_k,j_k}\ra\\
&=\la\oplus_{k=1}^{n_i}\sqrt{\lm_k}\xi_{i_k,j_k},(1\ten 1\ten a)(\lm_s\ten 1\ten\alpha_s)(\oplus_{k=1}^{n_i}\sqrt{\lm_k}\xi_{i_k,j_k})\ra,
\end{align*}
where $\xi_i:=\oplus_{k=1}^{n_i}\sqrt{\lm_k}\xi_{i_k,j_k}\in \oplus_{k=1}^{n_i}C_c(G,A)$ lies in the unit ball of the Hilbert $A$-module $L^2(G,\ell^2(A))$.

Finally, when $Z(A^{**})=Z(A)^{**}$, inspection of the proof shows that the $\xi_{i,j}$ from the Kaplansky density argument can be taken in $C_c(G,Z(A))$. In this case, $h_{i,j}(s)(a)=ak_{i,j}(s)$, where $k_{i,j}(s)=\la \xi_{i,j},(\lm_s\ten\alpha_s)\xi_{i,j}\ra$ is a continuous  compactly supported function $G\rightarrow Z(A)$ of positive type. It follows that the $h_i$ from the final convexity argument satisfy $h_i(s)(a)=ak_i(s)$ for some continuous compactly supported function $k_i:G\rightarrow Z(A)$ of positive type, which, by \cite[Proposition 2.5]{AD87} is necessarily of the form $k_i(s)=\la\xi_i,(\lm_s\ten\alpha_s)\xi_i\ra$ for some contractive net $(\xi_i)\subseteq L^2(G,Z(A))$. The norm density of $C_c(G,Z(A))$ inside $L^2(G,Z(A))$ then yields the claim. 
\end{proof}

\begin{remark} Contrary to the well-known group case ($A=\bC$), it is not clear whether every continuous completely positive Herz-Schur multiplier $h:G\rightarrow\mc{CB}(A)$ of compact support is necessarily of the form $h_i(s)(a)=\la\xi,(1\ten a)(\lm_s\ten\alpha_s)\xi\ra$ for some $\xi\in L^2(G,A)$. Indeed, this was already asked for discrete dynamical systems in \cite[Remark 4.29]{BC2}. If this were true, then the net $(\xi_i)$ in the conclusion of Proposition \ref{p:C*amen} can be taken in $C_c(G,A)$, and it would follow from the proof of Theorem \ref{t:C*amen} (see below) that amenability is equivalent to the 1-positive approximation property for arbitrary $(A,G,\alpha)$.
\end{remark}

The following lemmas will be used to establish Theorem \ref{t:posdef}, which, as a corollary, entails the equivalence of conditions (2) and (3) of Theorem \ref{t:C*amen}. The first is standard, and the second is surely known, but we include proofs for completeness.

\begin{lem}\label{l:posdef} Let $G$ be a locally compact group, and $f\in C_c(G)$. Then $\lm(f)\geq0$ if and only if $\Delta^{1/2}f$ is positive definite.
\end{lem}

\begin{proof}
	This follows from the identity $\la \Delta^{1/2} f, g^* \ast g \ra = \la \lm(f) (\Delta^{1/2}g)^{\vee}, (\Delta^{1/2}g)^{\vee} \ra$ for $f,g \in C_c(G)$, where the former pairing is the dual pairing $(B(G),C^*(G))$, the latter is the inner product on $L^2(G)$, and $(\Delta^{1/2}g)^\vee(t) = \sqrt{\Delta(t^{-1})}g(t^{-1})$ for $t \in G$.
\end{proof}

\begin{lem}\label{l:dense} Let $(A,G,\alpha)$ be a $C^*$-dynamical system. Then $\mathrm{span}\{f^*\star f\mid f\in C_c(G,A)\}$ is norm dense in $C_0(G,A)$.
\end{lem}

\begin{proof} Let $(f_i)$ be a bai for $\LO$ consisting of states in $C_c(G)$ whose support goes to $\{e\}$. Let $(a_j)$ be a bai for $A$, and let $f_{i,j}\in C_c(G,A)$ be $f_{i,j}(s)=f_i(s)\alpha_s(a_j)$. Then $(f_{i,j})$ is a bai for the convolution algebra $L^1(G,A)$ (see, e.g., \cite[Proposition 16.4.3]{Ren}). By density of $C_c(G)\ten A$ in $C_0(G,A)$ and a simple polarization argument, it suffices to show that $f_{i,j}\star (g\ten a)\rightarrow (g\ten a)$ uniformly in $C_0(G,A)$ for all $g\in C_c(G)$ and $a\in A$. 

First, $g\in C_c(G)$ is uniformly continuous, so 
\begin{equation}\label{e:l1}f_i\ast g\rightarrow g\end{equation}
uniformly, where $\ast$ denotes convolution in $L^1(G)$. Second, by norm continuity of $\alpha_t(a)$ at the identity, the standard argument shows that
\begin{equation}\label{e:l2}\int_G f_i(t)\norm{\alpha_t(a)-a} \ dt\rightarrow0.\end{equation}
Then (\ref{e:l1}) and (\ref{e:l2}) together with the fact that $a_ja\rightarrow a$ imply
\begin{align*}\norm{f_{i,j}\star(g\ten a)(s)-g\ten a(s)}&=\norm{\int_G f_{i}(t)(g(t^{-1}s)\alpha_t(a_ja) - g(s)a) \ dt}\\
&\leq\int_G|f_i(t)g(t^{-1}s)|\norm{\alpha_t(a_ja)-\alpha_t(a)} \ dt\\
& \ \ \ \ \ + \int_G|f_i(t)g(t^{-1}s)|\norm{\alpha_t(a)-a} \ dt\\
& \ \ \ \ \ + \bigg|\int_G f_i(t)(g(t^{-1}s)-g(s)) \ dt\bigg|\norm{a}\\
&\leq\norm{f_i}_1\norm{g}_\infty\norm{a_ja - a}\\
& \ \ \ \ \ + \norm{g}_\infty\int_G f_i(t)\norm{\alpha_t(a)-a} \ dt\\
& \ \ \ \ \ + \norm{a}|f_i\ast g-g|(s)\\
&\xrightarrow{i,j}0
\end{align*}
uniformly in $s$. Thus, $f_{i,j}\star (g\ten a)\rightarrow (g\ten a)$ uniformly in $C_0(G,A)$, and the claim is verified.
\end{proof}

\begin{thm}\label{t:posdef} Let $(A,G,\alpha)$ be a $C^*$-dynamical system and let $(h_i)$ be a bounded net of continuous, compactly supported, completely positive Herz-Schur multipliers. The following conditions are equivalent.
\begin{enumerate}
\item $\norm{h_i(s)(f(s))-f(s)}\rightarrow 0$ for every $f\in C_c(G,A)$, uniformly for $s$ in compact subsets of $G$;
\item $h_i(e)\rightarrow\id_A$ and $\Theta(h_i)\rightarrow\id_{G\ltimes A}$ in the respective point norm topologies.
\end{enumerate}
\end{thm}

\begin{proof} $(1)\Rightarrow(2)$: First, pick $g\in C_c(G)$ with $g(e)=1$. Given $a\in A$, applying condition (1) to $f=g\ten a$ at $s=e$ implies that $\norm{h_i(e)(a)-a}\rightarrow0$. 

Second, we have
$$\norm{\alpha(h_i(s)(f(s)))(\lm_s\ten 1)-\alpha(f(s))(\lm_s\ten 1)}=\norm{\alpha(h_i(s)(f(s))-f(s))}=\norm{h_i(s)(f(s))-f(s)}\rightarrow 0$$
for every $f\in C_c(G,A)$, uniformly for $s$ in compact subsets of $G$. Hence, by definition of $\Theta(h_i)$ we have
$$\norm{\Theta(h_i)((\alpha\times\lm(f)))-(\alpha\times\lm(f))}\leq\int_{\mathrm{supp(f)}}\norm{\alpha(h_i(s)(f(s)))(\lm_s\ten 1)-\alpha(f(s))(\lm_s\ten 1)} \ ds\rightarrow0$$
for every $f\in C_c(G,A)$. By boundedness of $(h_i)$, it follows that $\Theta(h_i)\rightarrow\id_{G\ltimes A}$ in the point norm topology.

$(2)\Rightarrow(1)$: Identify $A$ with $zA\subseteq A_\alpha''$, and identify the $C^*$-dynamical systems $(A,G,\alpha)\cong (zA,G,\overline{\alpha})$. We may also assume $A_\alpha''\subseteq\BH$ is standardly represented, so that $\alpha(x)=U^*(1\ten x)U$, for a unitary $U\in L^\infty(G)\oten B(H)$.

Using the standard implementation $U$ along with the commutation relation $U(\lm_s\ten 1)=(\lm_s\ten u_s)U$, for each $f\in C_c(G,A)$ we have
\begin{align*}
\int_G\lm_s\ten h_i(s)(f(s))u_s \ ds&=\int_G(1\ten h_i(s)(f(s)))(\lm_s\ten u_s) \ ds\\
&=\int_GU\alpha(h_i(s)(f(s)))U^*(\lm_s\ten u_s) \ ds\\
&=\int_GU\alpha(h_i(s)(f(s)))(\lm_s\ten 1)U^* \ ds\\
&=U\Theta(h_i)((\alpha\times\lm)(f))U^*\\
&\rightarrow U(\alpha\times\lm)(f)U^*\\
&=\int_G\lm_s\ten f(s)u_s \ ds,
\end{align*}
where the convergence is in the norm topology of $\mc{B}(L^2(G,H))$. Consequently, for any  $\eta\in H$, with $\om_\eta$ denoting the associated vector functional on $B(H)$,
\begin{align*}\int_G\la\eta, h_i(s)(f(s))u_s\eta\ra \lm_s\ ds &= (\id\ten\om_{\eta})\bigg(\int_G\lm_s\ten h_i(s)(f(s))u_s \ ds\bigg)\\
&\rightarrow(\id\ten\om_{\eta})\bigg(\int_G\lm_s\ten f(s)u_s \ ds\bigg)\numberthis \label{e:conveta}\\
&=\int_G\la\eta, f(s)u_s\eta\ra \lm_s\ ds,
\end{align*}
where the convergence is in $(C^*_\lm(G),\norm{\cdot})$ and is uniform for $\eta$ in bounded subsets of $H$. 

Let $f\in C_c(G,A)$ be positive in the sense that $f=f_0^*\star f_0$ in the convolution algebra $C_c(G,A)$. Then by positivity of $\Theta(h_i)$,
$$\int_G\lm_s\ten h_i(s)(f(s))u_s \ ds=U\Theta(h_i)((\alpha\times\lm)(f))U^*\geq0$$
so that 
$$\int_G\la\eta, h_i(s)(f(s))u_s\eta\ra \lm_s\ ds=\lm(v_{i,f,\eta})\geq0,$$
where $v_{i,f,\eta}(s)=\la\eta, h_i(s)(f(s))u_s\eta\ra$. Similarly,
$$\int_G\la\eta, f(s)u_s\eta\ra \lm_s \, ds=\lm(v_{f,\eta})\geq0,$$
where $v_{f,\eta}(s)=\la\eta, f(s)u_s\eta\ra$. By Lemma \ref{l:posdef} $w_{i,f,\eta}:=\Delta^{1/2}v_{i,f,\eta}$ and $w_{f,\eta}:=\Delta^{1/2}v_{f,\eta}$ are positive definite functions on $G$. Applying the convergence (\ref{e:conveta}) to $(\Delta^{1/2}g\ten 1)f\in C_c(G,A)$, for $g\in C_c(G)$, it follows that
$$\norm{\lm(w_{i,f,\eta}g)-\lm(w_{f,\eta}g)} \to 0, \ \ \ g\in C_c(G),$$
uniformly for $\eta$ in bounded subsets of $H$. 

We now show that $w_{i,f,\eta}\rightarrow w_{f,\eta}$ weak* in $B(G)$, and that $\norm{w_{i,f,\eta}}_{B(G)}\rightarrow \norm{w_{f,\eta}}_{B(G)}$, both uniformly in $\eta$. First, observe that $(w_{i,f,\eta})$ is bounded in $B(G)=C^*(G)^*$ uniformly in $\norm{\eta}$: since $w_{i,f,\eta}$ is positive definite, we have
\begin{align*}\norm{w_{i,f,\eta}}_{B(G)}&= w_{i,f,\eta}(e) \\
&= \la\eta,h_i(e)(f(e))\eta\ra\\
&\leq\norm{h_i(e)}_{cb}\norm{f(e)}\norm{\eta}^2
\end{align*}
Given $g\in C_c(G)$, pick $v\in A(G)$ with $v\equiv 1$ on $\mathrm{supp}(g)$, then
$$\la w_{i,f,\eta}-w_{f,\eta},g\ra=\la w_{i,f,\eta}-w_{f,\eta},vg\ra=\la\lm(w_{i,f,\eta}g)-\lm(w_{f,\eta}g),v\ra\rightarrow0.$$
Since the image of $C_c(G)$ under the universal representation of $G$ is dense in $C^*(G)$ and $(w_{i,f,\eta})$ is bounded in $B(G)=C^*(G)^*$ (uniformly in $\norm{\eta}$), we have $w_{i,f,\eta}\rightarrow w_{f,\eta}$ weak* in $B(G)$, uniformly for $\eta$ in bounded subsets.

The convergence $\norm{w_{i,f,\eta}}_{B(G)}\rightarrow \norm{w_{f,\eta}}_{B(G)}$ and its uniformity in $\eta$ follow from the point norm convergence $h_i(e)\rightarrow\id_A$:
$$\lim_i\norm{w_{i,f,\eta}}_{B(G)}=\lim_i\la\eta,h_i(e)(f(e))\eta\ra=\la\eta,f(e)\eta\ra=\norm{w_{f,\eta}}_{B(G)}.$$
Thus, in the notation of \cite{GL}, $w_{i,f,\eta}\rightarrow w_{f,\eta}$ in $(B(G),\tau_{nw^*})$, uniformly for $\eta$ in bounded subsets of $H$. By \cite[Theorem A]{GL}, it follows that $w_{i,f,\eta}\rightarrow w_{f,\eta}$ in the $A(G)$-multiplier topology, and therefore uniformly on compact sets, and the convergence is uniform for $\eta$ in bounded subsets of $H$. Thus, given $K\subseteq G$ compact and $\ep>0$, pick $i_\ep$ such that 
$$\sup_{s\in K}|w_{i,f,\eta}(s)-w_{f,\eta}(s)|<\frac{\ep}{\sup_{s\in K}\Delta^{-1/2}(s)}, \ \ \ i\geq i_\ep.$$
Then for all $i\geq i_\ep$
$$\sup_{s\in K}|v_{i,f,\eta}(s)-v_{f,\eta}(s)|=\sup_{s\in K}|\Delta^{-1/2}(s)||w_{i,f,\eta}(s)-w_{f,\eta}(s)|<\ep,$$
and $v_{i,f,\eta}\rightarrow v_{f,\eta}$ uniformly on compact sets, uniformly for $\eta$ in bounded subsets of $H$. 
In particular,
$$\sup_{\norm{\eta}\leq 2}|\la\eta,(h_i(s)(f(s))-f(s))u_s \eta\ra|=\sup_{\norm{\eta}\leq 2}|\la\eta,h_i(s)(f(s))u_s\eta\ra - \la \eta,f(s)u_s\eta\ra|\rightarrow0,$$
uniformly for $s$ in compact subsets of $G$. Hence, by polarization,
\begin{align*}\norm{h_i(s)(f(s))-f(s)}&=\norm{(h_i(s)(f(s))-f(s))u_s}\\
&=\sup_{\norm{\eta_1},\norm{\eta_2}\leq 1}|\la\eta_1,(h_i(s)(f(s))-f(s))u_s \eta_2\ra|\\
&\leq\frac{1}{4}\sum_{k=0}^3\sup_{\norm{\eta_1},\norm{\eta_2}\leq 1}|\la(\eta_1+i^k\eta_2),(h_i(s)(f(s))-f(s))u_s (\eta_1+i^k\eta_2)\ra|\\ &\rightarrow0
\end{align*}
for each $f\in C_c(G,A)$ of the form $f_0^*\star f_0$, uniformly for $s$ in compact subsets of $G$. By boundedness of $h_i(s)$ in $\mc{CB}(A)$, Lemma \ref{l:dense} and a standard $3\ep$-argument, it follows that
$$\norm{h_i(s)(f(s))-f(s)}\rightarrow0, \ \ \ f\in C_c(G,A),$$
uniformly for $s$ in compact subsets of $G$.

\end{proof}



Let $(A,G,\alpha)$ be a $C^*$-dynamical system. The space $L^2_c(G)\whten A_\alpha''\cong M_{|I|,1}(A_\alpha'')$, where $|I|$ is the dimension of $L^2(G)$, and is therefore a Hilbert $W^*$-module over $A_\alpha''$ in the canonical fashion \cite{Blech}. The next lemma is used to make sense of the ``diagonal'' action of $\LI\oten A_\alpha''$ from $L^2(G,A)$ into $L^2_c(G)\whten A_\alpha''$. 

\begin{lem}\label{l:diagonal} Let $G$ be a locally compact group and let $A$ be a $C^*$-subalgebra of a von Neumann algebra $M$. There exists a contraction
$$\pi:\LI\oten M\rightarrow\mc{CB}(L^2(G,A),L^2(G)_c\whten M)$$
such that for every $F\in\LI\oten M$, $\xi,\eta\in L^2(G)$, $a\in A$, and $\mu\in M_*$,
$$\la\eta\ten\mu,\pi(F)(\xi\ten a)\ra=\la(\om_{\eta,\xi}\ten\id)(F)a,\mu\ra=\int_G\xi(s)\overline{\eta}(s)\la \tilde{F}(s)a,\mu\ra \ ds.$$
\end{lem}

\begin{proof} Let $\pi_1:\LI\rightarrow\mc{CB}(L^2_c(G))$ and $\pi_2:M\rightarrow\mc{CB}(A,M)$ be the canonical maps given by left multiplication. Since both $\pi_1$ and $\pi_2$ are normal, their tensor product extends to the weak* spatial tensor product
$$\pi_1\ten\pi_2:\LI\oten M\rightarrow \mc{CB}(L^2_c(G))\oten \mc{CB}(A,M).$$
As $\mc{CB}(L^2_c(G))\cong B(L^2(G))$ \cite[Theorem 3.4.1]{ER} has the dual slice map property, and $B(L^2(G))\cong (L^2_c(G)\opten L^2_c(G)^*)^*$ (by \cite[Propositions 9.3.2, 9.3.4]{ER}), it follows that 
\begin{align*} \mc{CB}(L^2_c(G))\oten \mc{CB}(A,M)&=(L^2_c(G)\opten L^2_c(G)^* \opten A\opten M_*)^*\\
&\cong(L^2_c(G)\opten A \opten L^2_c(G)^*\opten M_*)^*\\
&=\mc{CB}(L^2_c(G)\opten A,L^2_c(G)\whten M),
\end{align*}
where the second isomorphism is simply the swap between the second and third legs, and the last equality uses the fact that $L^2_c(G)\oten M=L^2_c(G)\whten M$ \cite[Corollary 3.5]{BS}. Thus, composing $\pi_1\ten\pi_2$ with the swap we obtain a weak*-weak* continuous complete contraction
$$\pi:\LI\oten M\rightarrow\mc{CB}(L^2_c(G)\opten A,L^2_c(G)\whten M),$$
such that for every $F\in\LI\oten M$, $\xi,\eta\in L^2(G)$, $a\in A$, and $\mu\in M_*$,
$$\la\eta\ten\mu,\pi(F)(\xi\ten a)\ra=\la(\om_{\eta,\xi}\ten\id)(F)a,\mu\ra=\int_G\xi(s)\overline{\eta}(s)\la \tilde{F}(s)a,\mu\ra \ ds$$
(this is obvious if $F=f\ten x$, $f\in\LI$, $x\in M$, and the general formula follows by weak*-weak* continuity of $\pi$). Note that the latter equality is independent of the representing function $\tilde{F}$ of $F$. 

To finish, we observe that $\pi(F)$ extends to $L^2(G,A)$ with norm less than $\norm{F}$. Let $\xi_1,...,\xi_n\in L^2(G)$, $a_1,..,a_n\in A$, and let $(e_i)$ be an orthonormal basis of $L^2(G)$. A standard Hilbert space argument shows that
$$(\om_{\xi_k,\xi_l}\ten\id)(F^*F)=\sum_i(\om_{\xi_k,e_i}\ten\id)(F^*)(\om_{e_i,\xi_l})(F),$$
where the sum converges weak* for each $k$ and $l$. Then with $\xi=\sum_{k=1}^n\xi_k\ten a_k$, we have
\begin{align*}\norm{\pi(F)\xi}_{w^*h}^2&=\norm{\sum_i(e_i^*\ten\id)(\pi(F)\xi)^*(e_i^*\ten\id)(\pi(F)\xi)}\\
&=\norm{\sum_{k,l=1}^n\sum_i(e_i^*\ten\id)(\pi(F)(\xi_k\ten a_k))^*(e_i^*\ten\id)(\pi(F)(\xi_l\ten a_l))}\\
&=\norm{\sum_{k,l=1}^n\sum_i((\om_{e_i,\xi_k}\ten\id)(F)a_k)^*((\om_{e_i,\xi_l}\ten\id)(F)a_l)}\\
&=\norm{\sum_{k,l=1}^n\sum_ia_k^*(\om_{\xi_k,e_i}\ten\id)(F^*)(\om_{e_i,\xi_l}\ten\id)(F)a_l}\\
&=\norm{\sum_{k,l=1}^na_k^*(\om_{\xi_k,\xi_l}\ten\id)(F^*F)a_l}\\
&\leq\norm{F}^2\norm{\sum_{k,l=1}^na_k^*(\om_{\xi_k,\xi_l}\ten\id)(1)a_l}\\
&=\norm{F}^2\norm{\xi}_{L^2(G,A)}^2,
\end{align*}
where the inequality follows from positivity of the map $[(\om_{\xi_k,\xi_l}\ten\id)]:\LI\oten M\rightarrow M_n(M)$. Hence, $\pi(F)$ extends to a bounded linear map from $L^2(G,A)$ into $L^2_c(G)\whten M$. A similar argument shows that $\pi(F)$ is completely bounded on $L^2(G,A)$ with $\norm{\pi(F)}_{cb}\leq\norm {F}$.
\end{proof}

We now possess the ingredients to establish our main result of this section.

\begin{proof}[Proof of Theorem \ref{t:C*amen}] $(1)\Rightarrow(2)$ follows directly from Proposition \ref{p:C*amen}.

$(2)\Leftrightarrow(3)$ follows immediately from Theorem \ref{t:posdef}.

$(3)\Rightarrow(4)$ follows from the techniques used in the proof  \cite[Lemma 6.5]{ABF}, which, as shown in the proof of $(7)\Rightarrow(8)$ in \cite[Proposition 3.12]{BEW}, extend to the locally compact case. We outline the construction, referring the reader to the proof of \cite[Proposition 3.12]{BEW} for details. Throughout the argument we identify $A$ with $zA\subseteq A_\alpha''$.

Let $(\xi_i)\subset C_c(G,\ell^2(A))$ be a net from (3). Note that we may view $(\xi_i)$ inside $\ell^2_c\hten L^2(G,A)$, as
\begin{align*}\ell^2_c\hten L^2(G,A)&=\ell^2_c\hten(L^2(G)_c\hten A)=(\ell^2_c\hten L^2(G)_c)\hten A\\
&=(\ell^2\ten L^2(G))_c\hten A \ \ \ \ \textnormal{\cite[Proposition 9.3.5]{ER}}\\
&=(L^2(G)\ten\ell^2)\hten A\\
&=L^2(G)_c\hten (\ell^2_c\hten A)\\
&=L^2(G,\ell^2(A)),
\end{align*}
where the Hilbert $A$-module structure on the latter space is
$$\la\xi,\eta\ra=\int_G\la\xi(s),\eta(s)\ra \ ds, \ \ \ \xi,\eta\in L^2(G,\ell^2(A)).$$
Let $\Lambda=\{a\in A\mid 0\leq a\leq 1\}$, which forms a bai for $A$ under the natural ordering, and converges weak* to 1 inside $A_\alpha''$. Define $P_{i,a}:\LI\oten A_\alpha''\rightarrow A_\alpha''$ by
$$P_{i,a}(F)=\la(1\ten 1\ten a^{1/2})\xi_i,(1\ten F)(1\ten 1\ten a^{1/2})\xi_i\ra, \ \ \ F\in \LI\oten A_\alpha'',$$
where we write $F$ for the map $\pi(F):L^2(G,A)\rightarrow L^2_c(G)\whten A_\alpha''$ from Lemma \ref{l:diagonal}. Then $P_{i,a}$ is a completely positive contraction.

Suppose $A_\alpha''\subseteq\BH$ and let $K=\oplus^2_{a\in\Lambda} H$. Then with $P_i:=\oplus_{a} P_{i,a}$, we obtain a completely positive contraction from $\LI\oten A_\alpha''$ into $\mc{B}(K)$. Passing to a subnet, we may assume that $P_i$ converges to $P$ in the weak* topology of $\mc{CB}(\LI\oten A_\alpha'',\mc{B}(K))$. For each $a\in\Lambda$, let $P_a:\LI\oten A_\alpha''\rightarrow A_\alpha''$ be the compression of $P$ to the $a^{th}$ block, and let $Q_a:\LI\oten Z(A_\alpha'')\rightarrow A_\alpha''$ be the restriction of $P_a$. The same monotonicity argument from \cite[Lemma 6.5]{ABF} shows that for each positive $F\in \LI\oten Z(A_\alpha'')$, $(Q_a(F))$ is increasing in $a$, and hence by boundedness it converges weak*. Let $Q:\LI\oten Z(A_\alpha'')\rightarrow A_\alpha''$ be the resulting map. Using the fact that $a\mapsto 1\ten 1\ten a$ and $s\mapsto 1\ten \lm_s\ten \alpha_s$ is an equivariant representation of $(A,G,\alpha)$ on the direct sum $\oplus_{n=1}^\infty L^2(G,A)\cong\ell^2_c\hten L^2(G,A)$, it follows more or less verbatim from the proof of \cite[Lemma 6.5]{ABF} (see also \cite[Proposition 3.12]{BEW}) that $Q$ is a $G$-equivariant projection of norm one. Hence, $(A,G,\alpha)$ is von Neumann amenable.

$(4)\Rightarrow(1)$ follows immediately from Theorem \ref{t:Reiter}.

Finally, when $Z(A^{**})=Z(A)^{**}$, the particular conclusion from Proposition \ref{p:C*amen} yields the claim. 
\end{proof}

\begin{remark} In \cite[Definition 3.27]{BEW}, Buss, Echterhoff and Willett defined a $C^*$-dynamical system $(A,G,\alpha)$ to have the \textit{(wAP)} if there exists a bounded net $(\xi_i)\in C_c(G,A)\subseteq L^2(G,A)$ such that for all $\mu\in A^*_c$, and $a\in A$, 
$$\mu(\la\xi_i,(1\ten a)(\lm_s\ten\alpha_s)\xi_i\ra-a)\rightarrow0,$$
uniformly on compact subsets of $G$. This notion is a weakening of Exel and Ng's positive approximation property, \textit{a priori} unrelated to condition (3) of Theorem \ref{t:C*amen}, which is a specific instance of the positive weak approximation property of B\'{e}dos and Conti. However, it was shown that the wAP coincides with amenability  \cite[Theorem 3.28]{BEW}. Hence, it is equivalent to condition (3) of Theorem \ref{t:C*amen}.
\end{remark}

\begin{cor} Let $(A,G,\alpha)$ be a $C^*$-dynamical system such that $Z(A^{**})=Z(A)^{**}$. Then $(A,G,\alpha)$ is amenable if and only if it has the 1-positive approximation property.
\end{cor}

\begin{proof} The forward direction follows immediately from the special case of Theorem \ref{t:C*amen}. The reverse direction is always true, by \cite[Theorem 3.28]{BEW}.
\end{proof}

\begin{cor}\label{c:SA} A commutative $C^*$-dynamical system $(C_0(X),G,\alpha)$ is amenable if and only if it is strongly amenable.
\end{cor}

\begin{proof} Only one direction requires proof. If $(C_0(X),G,\alpha)$ is amenable, by the special case of Theorem \ref{t:C*amen} when $Z(A^{**})=Z(A)^{**}$, there exists a net $(\xi_i)$ in $C_c(G,C_0(X))$ whose corresponding positive type functions $h_i(s)=\la\xi_i,(\lm_s\ten\alpha_s)\xi_i\ra$ satisfy $\norm{h_i(e)}\leq 1$ and 
$$\norm{h_i(s)f-f}\rightarrow0, \ \ \ f\in C_0(X),$$
uniformly for $s$ in compact subsets of $G$. It follows that $h_i(s)\rightarrow 1$ strictly in $C_b(X)$, uniformly on compact subsets of $G$. Since the strict topology and the topology of uniform convergence on compacta agree on bounded subsets of $C_b(X)$ \cite[Theorem 1]{Buck}, the associated functions $h_i:G\times X\rightarrow\bC$ converge to 1 uniformly on compact subsets of $G\times X$. By norm density of $C_c(G)\ten C_c(X)$ in $L^2(G,C_0(X))$ we may assume without loss of generality that $\xi_i\in C_c(G)\ten C_c(X)$. Then the net $(\xi_i)$ satisfies the conditions of \cite[Proposition 2.5(2)]{AD}, hence $(G,X)$ is an amenable transformation group. 
\end{proof}

\section*{Acknowledgements}

The authors would like to thank Alcides Buss, Siegfried Echterhoff and Rufus Willett for helpful interactions and for keeping us updated on their recent work \cite{BEW}, which inspired a good portion of our results in section 4. We would also like to thank Mehrdad Kalantar for helpful discussions at various points during the project. The second author was partially supported by the NSERC Discovery Grant RGPIN-2017-06275.

\end{spacing}

\end{document}